\documentclass[reqno, letterpaper, oneside]{article}
\usepackage{amsmath,amsfonts,amssymb,amsthm}
\usepackage[final]{graphicx}

\usepackage{bbm}
\usepackage{setspace}
\usepackage{geometry}
\numberwithin{equation}{section}

\newtheorem{theorem}{Theorem}[section]
\newtheorem{lemma}[theorem]{Lemma}
\newtheorem{corollary}[theorem]{Corollary}
\newtheorem{remark}[theorem]{Remark}

\newcommand\E{{\mathbb E}}
\renewcommand\Pr{{\mathbb P}}
\newcommand\dto{\overset{\mathrm{d}}{\to}}

\newcommand{\cD}{{\mathcal D}}
\newcommand{\cE}{{\mathcal E}}

\newcommand{\cL}{{\mathcal L}}

\newcommand{\cP}{{\mathcal P}}

\newcommand\NN{\mathbb{N}}
\newcommand{\RR}{{\mathbb R}}

\newcommand{\pc}{p_c}

\newcommand{\dist}{\mathrm{d}}

\newcommand\floor[1]{\lfloor #1 \rfloor}
\newcommand\ceil[1]{\lceil #1 \rceil}
\renewcommand\le{\leqslant}
\renewcommand\ge{\geqslant}

\newcommand{\indic}[1]{\mathbbm{1}_{{\{{#1}\}}}}

\DeclareMathOperator*{\argmax}{argmax}

\newcommand{\goodchi}{\protect\raisebox{2pt}{$\chi$}}

\renewcommand{\epsilon}{\varepsilon}
\newcommand{\eps}{\varepsilon}

\long\def\symbolfootnote[#1]#2{\begingroup
\def\thefootnote{\fnsymbol{footnote}}\footnote[#1]{#2}\endgroup}

\newenvironment{romenumerate}[1][-10pt]{
\addtolength{\leftmargini}{#1}\begin{enumerate}
 }{\end{enumerate}}

\newcounter{thmenumerate}
\newenvironment{thmenumerate}
{\setcounter{thmenumerate}{0}%
 \def\item{\par
 \refstepcounter{thmenumerate}\textup{(\roman{thmenumerate})\enspace}}
}
{}

\newcommand{\dx}{\mathrm d}
\newcommand\set[1]{\ensuremath{\{#1\}}}
\newcommand\xpar[1]{(#1)}
\newcommand\bigpar[1]{\bigl(#1\bigr)}
\newcommand\Bigpar[1]{\Bigl(#1\Bigr)}
\newcommand\biggpar[1]{\biggl(#1\biggr)}
\newcommand\lrpar[1]{\left(#1\right)}
\newcommand\tooo{\to\infty}
\newcommand\dd{\,\dx}
\newcommand\gl{\lambda}
\newcommand\gd{\delta}
\newcommand\Bex{\mathcal{B}_{\mathrm{ex}}}
\newcommand\Enp{\E_{n,p}}
\newcommand\notarrow{\not\leftrightarrow}
\newcommand\DD{\operatorname{D_{n,p}}} 
\newcommand\gLL{\Lambda^{(\lambda)}}
\newcommand{\Holder}{H\"older}
\newcommand\bp{{\mathfrak{X}}}
\newcounter{case}
\newcommand\pfcase[1]{\refstepcounter{case}\smallskip\noindent\emph{Case
	\arabic{case}: #1} \noindent}

\newcommand\pfitemx[1]{\par#1:}
\newcommand\pfitemref[1]{\pfitemx{\ref{#1}}}
\newcommand\Bin{\mathrm{Bin}}

\begin{document}
\title{On the critical probability in percolation}
\author{Svante Janson%
\thanks{Department of Mathematics, Uppsala University, 
PO Box 480, SE-751~06 Uppsala, Sweden. 
E-mail: {\tt svante.janson@math.uu.se}. 
Part of the work was done during visits to the University of Cambridge and
to the
Isaac Newton Institute for Mathematical Sciences
during the programme 
Theoretical Foundations for Statistical Network Analysis
(EPSCR Grant Number EP/K032208/1)
and was partially supported by a grant from 
the Knut and Alice Wallenberg Foundation
and a grant from
the Simons foundation.}
\ and Lutz Warnke%
\thanks{Department of Pure Mathematics and Mathematical Statistics,
University of Cambridge, Wilberforce Road, Cambridge CB3\thinspace0WB, UK; 
{\em and\/}
School of Mathematics, Georgia Institute of Technology, Atlanta GA~30332, USA. 
E-mail: {\tt L.Warnke@dpmms.cam.ac.uk}. 
Part of the work was done while interning with the Theory Group at 
Microsoft Research, Redmond; LW thanks Yuval Peres for introducing him to the topic of this paper.}} 
\date{November 25, 2016}

\maketitle

\begin{abstract}
For percolation on finite transitive graphs, Nachmias and Peres suggested a 
characterization of the critical probability based on the logarithmic 
derivative of the susceptibility. 
As a first test-case, we study their suggestion for the Erd{\H{o}}s--R{\'e}nyi 
random graph~$G_{n,p}$, and confirm that the logarithmic derivative has the 
desired properties: 
(i)~its maximizer lies inside the critical window~$p=1/n+\Theta(n^{-4/3})$, and 
(ii)~the inverse of its maximum value coincides with the $\Theta(n^{-4/3})$--width 
of the critical window. 
We also prove that the maximizer is not located at $p=1/n$ or $p=1/(n-1)$, 
refuting a speculation of~Peres. 
\end{abstract}

\section{Introduction}
The percolation phase transition on finite graphs is one of the most 
intriguing and striking phenomena at the intersection of mathematical 
physics, combinatorics, and probability theory. 
The classical Erd{\H{o}}s--R{\'e}nyi random graph~$G_{n,p}$ is perhaps 
the most carefully studied reference model: 
as the edge probability~$p$ increases past the `critical probability' 
$\pc=1/n$, the global structure changes radically, from only small 
components to a single giant component plus small ones. 
More precisely, using the parametrization $p=1/n + \lambda_n n^{-4/3}$,
and for simplicity assuming $p =\Theta(1/n)$,  
by the inspiring work 
of Erd{\H{o}}s and R{\'e}nyi~\cite{ER1960}, Bollob{\'a}s~\cite{Bollobas1984}, 
{\L}uczak~\cite{Luczak1990}, and Aldous~\cite{Aldous1997},   
we nowadays distinguish three qualitatively different phases of~$G_{n,p}$. 
In the subcritical phase $\lambda_n \to -\infty$, the $r=\Theta(1)$ largest
components $C_1, \ldots, C_r$ are typically all of comparable size: 
$|C_1| \sim |C_2| \sim \cdots \sim |C_r| = \Theta(n^{2/3}\lambda_n^{-2}\log
|\lambda_n|) = o(n^{2/3})$. 
In the supercritical phase $\lambda_n \to \infty$, 
the largest component typically dominates all other components:  
$|C_2| \ll |C_1| = \Theta(\lambda_n n^{2/3})$. 
In the critical window $|\lambda_n| = O(1)$, the rescaled sizes $|C_1|/n^{2/3},
|C_2|/n^{2/3}, \cdots$ of the largest components converge in distribution to 
 non-degenerate random variables, i.e., they are not concentrated.

In the language of mathematical physics, $G_{n,p}$ interpreted as 
percolation on the complete $n$-vertex graph is a mean-field model.  
Hence, we expect that the percolation phase transition of 
many `high dimensional' finite graphs is similar, with the hypercube 
and various tori being examples of great interest 
(see, e.g.,~\cite{AKS82,BKL92,vdHS06,BCHSS1,BCHSS2,BCHSS3,vdHN2012}). 
To fix notation, 
we assume that $G$ is a given transitive $n$-vertex graph, and
we write $G_p \subseteq G$ for the binomial random subgraph where 
each edge is included independently with probability~$p$. 
As pointed out by Nachmias and Peres~\cite{NP08}, in this general 
percolation setting it is a challenging problem to find a good 
definition of the critical probability~$p_c$,
such that for a suitable critical window around $p_c$,
for example, the size of the largest component is not concentrated.

The folklore average degree heuristic $p_c=1/(\mathrm{deg}_G(v)-1)$ 
is a natural  first guess 
(the graph $G$ is assumed to be transitive and thus regular, 
so the choice of the vertex~$v$ does not matter).
For the  hypercube with vertex set $\{0,1\}^m$,
and thus degree $m$, 
Ajtai, Koml{\'o}s and Szemer{\'e}di \cite{AKS82}
showed that there is a critical threshold $(1+o(1))/m$; this was sharpened by
Bollob{\'a}s, Kohayakawa, and {\L}uczak~\cite{BKL92}, who raised the
question whether the critical probability might be exactly 
$1/(m-1)$.
However, Borgs, Chayes, van der Hofstad, Slade and
Spencer~\cite{BCHSS1,BCHSS2} and 
van der Hofstad and Nachmias~\cite{vdHN2012,vdHN2014}
have shown that there is 
critical window of width $\Theta(n^{-1/3}p_c)=\Theta(2^{-m/3}/m)$
about a critical probability $p_c$, which by 
van der Hofstad and Slade~\cite{vdHS06} 
satisfies $p_c = 1/(m-1) + 3.5 m^{-3} + O(m^{-4})$;
since the width of the window is $o(m^{-3})$, 
the value $1/(m-1)$ is outside the critical window.

A more sophisticated suggestion for the critical probability 
was pioneered by Borgs, Chayes, van der Hofstad, Slade and
Spencer~\cite{BCHSS1,BCHSS2} (and used for the hypercube result just described).
They essentially proposed to define~$\pc=\pc(G)$ as the unique solution 
to the polynomial equation
\begin{equation}\label{eq:pc:mf}
\goodchi_G(p) := \E_p |C(v)| = n^{1/3} ,
\end{equation}
where the susceptibility $\goodchi_G(p)$ denotes the expected size of the
component~$C(v)$ containing a fixed vertex~$v$ in~$G_p$. 
(This is a widely studied key parameter in percolation theory and random
graph theory, see, 
e.g., \cite{AizenmanNewman1984,Grimmett1999,JansonLuczak2008,LNNP2014,RWapsubcr,RWapbsr}. 
Since~$G$ is assumed to be transitive, the choice of~$v$ does not matter.)  
The aforementioned technical definition is guided by Erd{\H{o}}s--R{\'e}nyi 
mean-field type behaviour. Indeed, in the subcritical phase we expect that~$C(v)$ 
closely mimics a subcritical branching process, which suggests that 
typically $|C_1| \approx (\goodchi_G(p))^2$ up to logarithmic corrections 
(see, e.g., Section 1.2 in~\cite{vdHL10} or Proposition~5.1 in~\cite{AizenmanNewman1984}). 
Furthermore, in the supercritical phase we expect that the largest 
component dominates all other components, which by transitivity 
suggests that $\goodchi_G(p) \approx \E_p |C_1|^2/n$. Assuming 
that inside the critical window we can observe subcritical and 
supercritical features, it thus seems plausible that the critical 
probability should roughly satisfy 
$\goodchi_G(p) \approx \E_p |C_1|^2/n \approx \goodchi_G(p)^4/n$,  
motivating the choice of equation~\eqref{eq:pc:mf}.  
Borgs et al.~\cite{BCHSS1,BCHSS2} showed that (a minor variant of) the
discussed definition is very useful in combination with the so-called finite
triangle condition:  they recovered many
Erd{\H{o}}s--R{\'e}nyi features under such generic mean-field assumptions  
(see~\cite{vdHN2012,vdHN2014} for some more recent developments).

As pointed out by Peres~\cite{Peres2012}, the suggestion of 
Borgs et al.~\cite{BCHSS1,BCHSS2} builds the mean-field scaling 
$\Theta(n^{1/3})$ into the definition of the critical probability.  
It would be desirable to have a useful general definition that recovers this
scaling for $n$-vertex mean-field graphs $G=G_n$,
rather than having separate definitions for each different scaling 
behaviour (or, in mathematical physics jargon, for each `universality class').
With this aim in mind, Nachmias and Peres~\cite{NP08} 
suggested to define~$\pc=\pc(G)$ as the value of~$p$ which maximizes 
the logarithmic derivative 
\begin{equation}\label{eq:pc:ld}
\frac{\dx}{\dx p} \log \goodchi_G(p) = \frac{\frac{\dx}{\dx p}\E_p |C(v)|}{\E_p |C(v)|} .
\end{equation}
To motivate this definition, note that by the Margulis--Russo 
formula~\cite{Margulis,Russo} the derivative $\frac{\dx}{\dx p}\E_p |C(v)|$ 
intuitively counts the expected (weighted) number of edges of $G_p$ which 
can affect the size of~$|C(v)|$, see also Section~\ref{sec:maxder}. 
In other words, $\pc$ equals the probability where the addition of a 
random edge has maximum relative impact on the component size~$|C(v)|$. 
Denoting the maximum value of~\eqref{eq:pc:ld} by~$M=M(G)$, 
Warnke~\cite{Warnke2012} conjectured that for mean-field graphs~$G$
the width of the critical window is of order~$\Theta(1/M)$. 
This is motivated by the fact that 
$\log(\goodchi_G(p_2)/\goodchi_G(p_1)) = \int_{p_1}^{p_2} \frac{\dx}{\dx p} \log \goodchi_G(p) dp \le (p_2 - p_1) M$ 
entails that the susceptibility satisfies 
$\goodchi_G(p_2) = \Theta(\goodchi_G(p_1))$ for $p_2-p_1 = O(1/M)$.

\subsection{Main results}
In this paper we investigate, as a first test-case, the suggested 
definition of Nachmias and Peres~\cite{NP08} for the Erd{\H{o}}s--R{\'e}nyi 
random graph~$G_{n,p}$ i.e., the case $G=K_n$ (as proposed by Peres~\cite{Peres2012}). 
Here our first main result confirms that their definition of 
the critical probability~$\pc$ has the desired properties, i.e., that 
for $G_{n,p}=(K_n)_p$ the logarithmic derivative 
$\frac{\dx}{\dx p}\log \goodchi_{K_n}(p)$ satisfies the following:%
\begin{romenumerate}%
\item its maximizer lies inside the critical window $p=1/n + O(n^{-4/3})$, 
and  
\item the inverse of its maximum value coincides with the $\Theta(n^{-4/3})$--width	
of the critical window.  
\end{romenumerate}
\begin{theorem}[Maximizer of the logarithmic derivative for $G_{n,p}$]%
\label{thm:main}
We have 
\begin{gather}
\label{main:p}
\biggl|\argmax_{p \in (0,1) }\frac{\dx}{\dx p} \log \goodchi_{K_n}(p) - \frac{1}{n}\biggr| = O(n^{-4/3}), \\ 
\label{main:max}
\max_{p \in (0,1)} \frac{\dx}{\dx p} \log \goodchi_{K_n}(p) = \Theta(n^{4/3}) .
\end{gather}
\end{theorem}
\begin{remark}
Theorem~\ref{thm:main2} shows that~\eqref{main:p} remains valid 
with $O(n^{-4/3})$ replaced by $\Theta(n^{-4/3})$. 
\end{remark}
Having established the qualitative behaviour of the logarithmic derivative for $G_{n,p}$, 
it is intriguing to investigate the finer scaling behaviour inside critical window. 
By symmetry considerations it might be tempting to believe that $p=1/n$ or
$p=1/(n-1)$ could be the maximizer of $\frac{\dx}{\dx p} \log
\goodchi_{K_n}(p)$, as speculated by~Peres~\cite{Peres2012}.  
Our second main result refutes this tantalizing belief, instead
strengthening the general feeling that $\lambda=0$ is no special point 
inside the critical window of form $p=1/n + \lambda n^{-4/3}$.  
\begin{theorem}[Scaling inside the critical window of $G_{n,p}$]%
\label{thm:main2}
Given $\lambda \in \RR$, for $p =1/n + \bigl(\lambda+o(1)\bigr) n^{-4/3}$ 
we have, as $n \to \infty$, 
\begin{align}
\label{main:conv}
\frac{\goodchi_{K_n}(p)}{n^{1/3}} &\to f(\lambda),\\
\label{main:der}
\frac{\frac{\dx}{\dx p} \log \goodchi_{K_n}(p)}{n^{4/3}} 
& \to \frac{\dx}{\dx \lambda} \log f(\lambda)
>0, 
\end{align}
where the infinitely differentiable function~$f=f_2:\RR \to (0,\infty)$ is
defined in~\eqref{def:Lambda}--\eqref{def:fk}.
Moreover, if $p=1/n+\lambda n^{-4/3}$, then
the convergence in \eqref{main:conv}--\eqref{main:der} is uniform for 
$\lambda$ in any compact interval $[\lambda_1,\lambda_2]\subset\RR$. 
Furthermore, 
\begin{equation}\label{eq:main20}
\frac{\dx^2}{\dx \lambda^2}\log f (0) \neq 0 .
\end{equation}
\end{theorem}
The definition of the function~$f$ appearing in Theorem~\ref{thm:main2} is
quite involved, since it intuitively needs to capture the contribution of
components with arbitrary numbers of cycles.
It is easy to find asymptotics as $\gl\to\pm\infty$; we have 
$f(\gl)\sim|\gl|^{-1}$ as $\gl\to-\infty$ and
$f(\gl)\sim 4\gl^{2}$ as $\gl\to+\infty$; hence
$\log f(\gl)=-\log|\gl|+o(1)$ as $\gl\to-\infty$ and
$\log f(\gl)=2\log\gl+O(1)$ as $\gl\to+\infty$;
furthermore, 
$\frac{\dx}{\dx \lambda} \log f(\lambda)=O(1/|\gl|)$ for all $\gl\in\mathbb
R$,
see Appendix~\ref{AppAgl} for proofs. 
Theorem~\ref{thm:main2} also extends to convergence of higher derivatives, 
see Appendix~\ref{AppAhigher}.

It would be interesting to know whether the logarithmic 
derivative~$\frac{\dx}{\dx \lambda} \log f(\lambda)$ has a unique 
maximizer~$\gl^*$, and whether it is unimodal.
Figure~\ref{fig:plots} below (which is obtained by numerical integrations)
suggests that this is the case, with~$\gl^*\approx1$ 
(we conjecture~$\gl^*> 1$ based on our limited precision numerical data).

\begin{figure}[h] 
\centering
\vspace{-0.125em}
\includegraphics[width=2.75in]{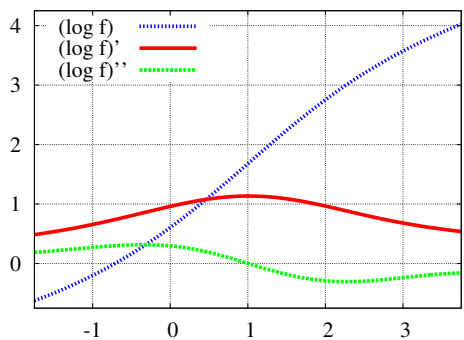}\vspace{-0.75em}
\caption{\label{fig:plots} Plot of the functions $\log f(\lambda)$, 
$\frac{\dx}{\dx \lambda} \log f(\lambda)$ and $\frac{\dx^2}{\dx \lambda^2} \log f(\lambda)$ 
for $\lambda \in [-1.75,3.75]$, 
where~$f$ is as in Theorem~\ref{thm:main2}. 
It provides some evidence for our belief that 
$\frac{\dx}{\dx \lambda} \log f(\lambda)$ has a unique 
maximizer~$\lambda^* \approx 1$.}%
\end{figure}

The high-level structure of our proofs is as follows. 
For Theorem~\ref{thm:main} our starting point is the Margulis--Russo
Formula, which allows us to write $\frac{\dx}{\dx p}\goodchi_{K_n}(p)$ in
terms of sums involving the squared component sizes of~$G_{n,p}$. Using
ideas from random graph theory we then estimate these sums, combining
correlation inequalities and the `symmetry rule' 
(also called `discrete duality principle') 
with results for the largest component and the susceptibility of $G_{n,p}$,
which eventually implies~\eqref{main:p}--\eqref{main:max}; see
Section~\ref{sec:maxder}.   
For Theorem~\ref{thm:main2} with $p=1/n + \bigl(\lambda
+o(1)\bigr) n^{-4/3}$, our starting point is the well-known fact that
$X_{n,2}=\sum_{j \ge 1}|C_j|^2/n^{4/3} \dto W_{\lambda,2}$ for some random
variable~$W_{\lambda,2}$. Using technical arguments we then justify taking
expectations and derivatives, which in view of $\goodchi_{K_n}(p)/n^{1/3} = \E_p
X_{n,2}$ eventually establishes~\eqref{main:conv}--\eqref{main:der}
with $f(\lambda)=\E W_{\lambda,2}$; see Section~\ref{sec:convergence}. For
inequality~\eqref{eq:main20} we show that~$f$ and its derivatives 
can be computed at~$\lambda=0$ by series expansions (exploiting recursive 
formulas for the area under a normalized Brownian excursion). 
Since these series converge exponentially, we can then 
numerically verify~\eqref{eq:main20} by finite truncation;
see Section~\ref{sec:zero}.

\subsection{Remarks on some other graphs}
In the present paper we discuss only the Erd{\H{o}}s--R{\'e}nyi random 
graph $G_{n,p}=(K_n)_p$, i.e.,  
percolation on the complete $n$-vertex graph.
In particular,  
Theorem~\ref{thm:main} shows that the definition of the critical 
probability~$\pc$ suggested by Nachmias and Peres~\cite{NP08,Peres2012} `works' 
in this case.
It is an interesting open problem to establish analogous results
for other finite graphs.

For example, consider
 again the hypercube with vertex set $\{0,1\}^m$ discussed above,
see \cite{BCHSS1,BCHSS2,vdHN2012,vdHN2014}. 
In the subcritical phase $p = (1-\eps)\pc$ with 
$\eps^3n \to\infty$,
\cite[Proposition~A.1]{BCHSS1} combined with 
\cite[Theorem 1.3 and Theorem~1.5]{BCHSS2} show that
$\frac{\dx}{\dx p} \goodchi_G(p) \sim m (\goodchi_G(p))^{2}$
and $\goodchi_G(p) \sim \eps^{-1}$,
and thus
\[
\frac{\dx}{\dx p} \log \goodchi_G(p) \sim m \goodchi_G(p) = \Theta(\eps^{-1} m) 
=o\bigl(n^{1/3} m\bigr).
\]
In the supercritical phase $p=(1+\eps)\pc$ with
$\eps^3n \to \infty$ and $\eps=\eps(n) \to 0$, we have $\goodchi_G(p) \sim 4
\eps^2 n$ 
according to
\cite[Theorem~1.1]{vdHN2012}; hence
it is natural to conjecture
that 
the logarithmic derivative satisfies 
\begin{equation*}\label{eq:logder:MFG}
\frac{\dx}{\dx p} \log \goodchi_G(p) 
= \pc^{-1}\frac{\dx}{\dx \eps} \log \goodchi_G(p) 
\approx  \frac{2}{\eps \pc} = \Theta(\eps^{-1}m)   
=o\bigl(n^{1/3} m\bigr) ,
\end{equation*}
in the supercritical phase too, 
and, moreover, that the logarithmic derivative has a maximum of order
$\Theta(n^{1/3}m)$ which is attained inside the critical window. 
Proving this, however, remains a challenging problem.

Another important example would be random $d$-regular graphs
with~$d=d(n) \to \infty$. 

Moreover, it would be conceptually very interesting to start with the maximizer 
of~\eqref{eq:pc:ld} and then derive properties of the phase transition of~$G_p$ 
(rather than, as in this paper, using known results for~$G_{p}$ to verify 
properties of the maximizer). 

It also seems highly desirable to better understand the critical
probability~$p_c$ for finite transitive graphs~$G$ 
which do not exhibit the mean-field behavior of the complete graph~$K_n$ 
or the hypercube~$\{0,1\}^m$.
Here the perhaps simplest example is percolation on the $n$-vertex 
cycle, $n \ge 3$, for which it is not difficult to check that 
there are three different phases: 
(i)~for $1-p = o(n^{-1})$ we typically have $|C_1| = n$, 
(ii)~for $1-p = \omega(n^{-1})$ we typically have $|C_1| = o(n)$, 
and (iii)~for~$1-p = \Theta(n^{-1})$ the rescaled sizes $|C_1|/n,
 \cdots, |C_r|/n$ of the $r=\Theta(1)$ largest components are 
not concentrated. 
Hence the critical window is parametrized by 
$p = 1-\lambda_n n^{-1}$ with $\lambda_n = \Theta(1)$. 
For $p \in (0,1)$ it is routine to see that 
\begin{align*}
\E_p |C(v)| & = 1 + \sum_{1 \le j < n}(p^{j}+p^{n-j}-p^n) = 1 + \sum_{1 \le j < n}(2p^{j}-p^{n}) , \\
\frac{\dx}{\dx p} \E_p |C(v)| &= \sum_{1 \le j < n}(2jp^{j-1}-np^{n-1}) = \sum_{1 \le j < n}2jp^{j-1}(1-p^{n-j}) .
\end{align*}
A short calculation shows that $\E_p |C(v)| = \Theta(n^{1/3})$ 
implies $p = 1-\Theta(n^{-1/3})$. Furthermore, 
$\frac{\dx}{\dx p} \log \E_p |C(v)| = \Theta(n)$ 
for $1-p = \Theta(n^{-1})$, and 
$\frac{\dx}{\dx p} \log \E_p |C(v)| = 
\Theta(\min\{(1-p)n^2,(1-p)^{-1}\}) = o(n)$ otherwise. 
For the critical probability~$\pc$ of $n$-vertex cycles, 
it follows that the mean-field definition of Borgs et al.\ fails 
(as expected, since cycles are not `high dimensional'). 
By contrast, the definition based on the maximizer of the 
logarithmic derivative of the susceptibility does correctly 
predict $\pc = 1- \Theta(n^{-1})$ and the $\Theta(n^{-1})$--width 
of the critical window, supporting the hope that this 
definition might work beyond the mean-field case.

\subsection{Some notation} 
For emphasis, we will often use $\Pr_{n,p}$ and $\E_{n,p}$ for
probability and expectation with respect to $G_{n,p}$. 
We let~$C_i$ denote the components of $G_{n,p}$ in order of decreasing
sizes, $|C_1| \ge |C_2| \ge \cdots$ (resolving ties by taking the component 
with the smallest vertex label first, for definiteness). 
Finally, convergence in distribution is denoted $\dto$, 
and unspecified limits are as $n\tooo$.

\section{Maximizer of the logarithmic derivative}\label{sec:maxder}
In this section we prove Theorem~\ref{thm:main}. Our arguments combine the 
Margulis--Russo formula with results and ideas from random graph theory. 
For mathematical convenience we shall work with the `rescaled' 
susceptibility parameters
\begin{align}
\label{def:S}
S(G_{n,p}) &:= \sum_{v \in [n]} |C(v)| 
= \sum_i |C_i|^2 ,
\end{align}
where the component $C(v)$ is with respect to~$G_{n,p}$,
as usual, and 
\begin{align}\label{def:Sn}
  S_n(p) & := \E S(G_{n,p}) = 
\E_{n,p}\Bigpar{\sum_i |C_i|^2}.
\end{align}
Recall that $\goodchi_{K_n}(p):=\E_{n,p}|C(v)|$, which is the same for 
every $v\in[n]$ by symmetry, and thus by \eqref{def:S}--\eqref{def:Sn} 
\begin{align}\label{Sn-chi}
  S_n(p) & =  n \goodchi_{K_n}(p) ,
\end{align}
which implies
\begin{align}
\label{eq:logder:chi:S}
\frac{\dx}{\dx p} \log \goodchi_{K_n}(p) = \frac{\dx}{\dx p} \log S_n(p) . 
\end{align}
Theorem~\ref{thm:main} follows from equation~\eqref{eq:logder:chi:S} and 
inequalities~\eqref{rough:upper}--\eqref{rough:lower} of
Theorem~\ref{thm:rough} below. 
(In fact, in the lower bound~\eqref{rough:lower}, it suffices to consider,
for example, $\gl=0$.)
\begin{theorem}[Bounds for the logarithmic derivative]%
\label{thm:rough}
There is a constant $C>0$ such that, for all $n \ge 1$ and~$p \in (0,1)$,  
\begin{equation}\label{rough:upper}
\frac{\dx}{\dx p} \log S_n(p) \le C \cdot \min\bigl\{|p-1/n|^{-1}, \: n^{4/3}\bigr\} .
\end{equation}
Furthermore, for every $\lambda \in \RR$ 
there is a constant $D_{\lambda}>0$ such that, for all 
$n \ge 2$ and $p=1/n+ \lambda n^{-4/3}\in(0,1)$, 
\begin{equation}\label{rough:lower}
\frac{\dx}{\dx p} \log S_n(p) \ge D_{\lambda} n^{4/3} .
\end{equation}
\end{theorem}
The remainder of this section is devoted to the proof of 
Theorem~\ref{thm:rough}, and we start by studying a combinatorial 
form of $\frac{\dx}{\dx p}S_n(p)$. 
Writing $v \leftrightarrow w$ for the event that~$v$ and~$w$ are 
connected (which trivially holds if $v=w$), note that 
$S(G)= \sum_{v,w \in V(G)} \indic{v \leftrightarrow w}$
and thus, by taking the expectation, see \eqref{def:Sn},
\begin{equation}\label{eq:S:sum}
S_n(p) = \sum_{v,w \in [n]} \Pr_{n,p}(v \leftrightarrow w) .
\end{equation}
We now record the following simple monotonicity property,
which is obvious from \eqref{eq:S:sum}. 
\begin{lemma}\label{lem:Sn}
If $p \le p'$ and $n \le n'$, then $S_n(p) \le S_{n'}(p')$.  
\qed
\end{lemma}
We say that an edge $e \in E(K_n)$ is \emph{pivotal} for $v \leftrightarrow w$, 
if~$v \leftrightarrow w$ in~$G_{n,p}+e$ and~$v \not\leftrightarrow w$ in~$G_{n,p}-e$ 
(i.e., in the possibly modified graphs where $e$ is added and removed, respectively). 
Recalling the form of~\eqref{eq:S:sum}, for $p \in (0,1)$ the Margulis--Russo 
Formula~\cite{Margulis,Russo} gives 
\begin{equation}\label{eq:S:der}
\frac{\dx}{\dx p}S_n(p) = \sum_{v,w \in [n]} \frac{\dx}{\dx p}\Pr_{n,p}(v \leftrightarrow w) = \sum_{v,w \in [n]} \sum_{e \in E(K_n)} \Pr_{n,p}(\text{$e$ is pivotal for $v \leftrightarrow w$}) .
\end{equation}
Let $\cP_{e,v,w}$ denote the event that (i)~$e \not\in G_{n,p}$ and (ii)~$e$~is 
pivotal for $v \leftrightarrow w$. 
Since being pivotal does not depend on the status of~$e$, it follows that 
\begin{equation}\label{eq:S:der:Q}
\frac{\dx}{\dx p}S_n(p) 
= \frac{\E_{n,p}\bigl(\sum_{v,w \in [n]} \sum_{e \in E(K_n)} \indic{\cP_{e,v,w}} \bigr)}{1-p} .
\end{equation}
An edge not present in $G_{n,p}$ is pivotal for $v \leftrightarrow w$ if and 
only if one of its endpoints is in $C(v)$ and the other is in $C(w) \neq C(v)$. 
Hence $\sum_{e \in E(K_n)} \indic{\cP_{e,v,w}} = \indic{C(v) \neq C(w)}|C(v)||C(w)|$. 
Consequently,
\begin{equation}\label{eq:Q:sum}
\sum_{v,w \in [n]} \sum_{e \in E(K_n)} \indic{\cP_{e,v,w}} = 
\sum_{v \in [n]} |C(v)| \sum_{w \not\in C(v)} |C(w)|
= \sum_{i \neq j}|C_i|^2|C_j|^2 ,
\end{equation}
and thus, by \eqref{eq:S:der:Q},
\begin{equation}\label{eq:SQ:sum}
\frac{\dx}{\dx p}S_n(p) 
=\frac{\E_{n,p}\lrpar{
\sum_{v \in [n]} |C(v)| \sum_{w \not\in C(v)} |C(w)|}}
{1-p}
=\frac{\E_{n,p}\lrpar{ \sum_{i \neq j}|C_i|^2|C_j|^2 }}
{1-p},
\end{equation}
which eventually allows us to bring random graph theory into play.

\subsection{Upper bounds}\label{sec:maxder:sub}
In this subsection we prove the upper bound~\eqref{rough:upper} 
from Theorem~\ref{thm:rough}. 

We shall use some more or less well-known results for the 
susceptibility
and the size of the largest component of~$G_{n,p}$ 
in near-critical cases, which we state as the following theorem.
(See, e.g., \cite{BCHSS1,JansonLuczak2008,LNNP2014,BR2014} for similar or
related results.) 

\begin{theorem}
\label{thm:RG}
\begin{romenumerate}
\item \label{T:RGa}
There is a constant $D >0$ such that, for all $n \ge 1$,
$p \in [0,1]$, and $\eps>0$, 
\begin{alignat}{2}
S_n(p) &\le \eps^{-1} n\qquad&&\text{if $np \le 1-\eps$},
\label{eq:S-}\\ 
\label{eq:S+}
S_n(p) &\le 
D \eps^2 n^2 \qquad&&\text{if $np \le 1+\eps$ and $\eps^3 n \ge 1$}.
\end{alignat}
\item \label{T:RGb}
For any $A>0$ there are constants $a,B,n_0>0$ such that, 
for all $n \ge n_0$, $p \in [0,1]$, $\eps \in (0,A]$, and 
$\delta \in (0,1/2]$ satisfying $np=1+\eps$ and $\delta^2\eps^3n \ge B$, 
\begin{equation}\label{C1:supcr:tail}
\Pr_{n,p}\bigl(\bigl| |C_1| - \rho(\eps) n \bigr| \ge \delta \rho(\eps) n\bigr) \le e^{-a \delta^2 \eps^3 n} , 
\end{equation}
where $\rho(\eps)>0$ is the positive solution to 
$1-\rho(\eps) = e^{-(1+\eps) \rho(\eps)}$. 
\item \label{T:RGc}
Furthermore, for any $A>0$ there are constants $\delta \in (0,1/2)$ 
and $c>0$ such that, for all $\eps \in (0,A]$, 
\begin{equation}
\label{C1:supcr:alpha}
0 < \bigl(1-(1-\delta)\rho(\eps)\bigr) \cdot (1+\eps) \le 1-c \eps. 
\end{equation}
\end{romenumerate}
\end{theorem}
\begin{proof}
The subcritical upper bound \eqref{eq:S-} for the susceptibility
is simple and well-known.
The supercritical upper bound \eqref{eq:S+} is intuitively clear, since in
the supercritical range, the susceptibility ought to be dominated by
$\Enp|C_1|^2$ and $|C_1|$ is with high probability $\Theta(n\eps)$ when
$np=1+\eps$. However, we are unaware of a reference which contains a 
short proof of~\eqref{eq:S+}, and thus for completeness we give in 
Appendix~\ref{sec:susc} proofs of both upper 
bounds~\eqref{eq:S-}--\eqref{eq:S+} for the susceptibility. 

The tail bound~\eqref{C1:supcr:tail} follows 
from~\cite[Theorem~4, (10) and Remark 3]{BR2014}.

The estimate~\eqref{C1:supcr:alpha} follows 
for small $\eps$, say $\eps\le\eps_0$, from the fact that
$\rho(\eps)=2\eps+o(\eps)$ as 
$\eps\to0$,
and for $\eps\in[\eps_0,A]$ from the fact that (with $\rho=\rho(\eps)$)
$(1-\rho)(1+\eps)=-(1-\rho)\log(1-\rho)/\rho<1$ for $\eps>0$ together with
the continuity of $\rho(\eps)$.
(Cf.~e.g.~\cite[Lemma A.2]{JansonLuczak2008}.)
\end{proof}

\begin{corollary}\label{C:largep}
There are constants $n_0,\pi_0,b>0$ such that, for all $n \ge n_0$ and 
$p \in [0,1]$ satisfying $np \ge \pi_0$, we have 
$\Pr(|C_1| \le n/2) \le e^{-b n}$ and $S_n(p) \ge n^2/8$.
\end{corollary}

\begin{proof}
Choose $\eps$ such that $\rho(\eps)\ge 3/4$ and let $\pi_0:=1+\eps$. 
Then the tail estimate~\eqref{C1:supcr:tail} and monotonicity yield 
$\Pr_{n,p}(|C_1| \le n/2) \le \Pr_{n,\pi_0/n}(|C_1| \le n/2) \le e^{-b n}$. 
The second conclusion follows from
$S_n(p) \ge \E |C_1|^2 \ge (n/2)^2 \Pr(|C_1| \ge n/2)$.
\end{proof}
We next prove two convenient auxiliary estimates.
\begin{lemma}\label{lem:sum:upper}
For all $n \ge 1$ and $p \in [0,1]$,  
\begin{align}
\label{eq:sum:upper} 
\E_{n,p}\Bigpar{\sum_{i \neq j}|C_i|^2 |C_j|^2} 
& \le S_n(p)^2 = \Bigl(\E_{n,p}\bigl(\sum_{i}|C_i|^2\bigr)\Bigr)^2 , \\
\label{eq:sum:upper:cross}
\E_{n,p}\Bigpar{\sum_{i,j} |C_i|^2|C_j|^2} & \le \Bigl( S_{n}(p) + 3 \bigl[n^{-1} S_n(p)\bigr]^4\Bigr) \cdot S_{n}(p) .
\end{align}
\end{lemma}
\begin{proof}
We start with~\eqref{eq:sum:upper} and fix any vertex $v \in [n]$. 
Conditioning on the vertex set of~$C(v)$ in~$G_{n,p}$, 
the remaining graph with vertex set $[n]\setminus C(v)$ has the same 
distribution as $G_{n-|C(v)|,p}$ (up to relabeling of the vertices). 
Since $S_{n-|C(v)|}(p) \le S_n(p)$ by Lemma~\ref{lem:Sn}, 
using~\eqref{def:S} it follows that 
\begin{equation}\label{logder:upper:cond}
\begin{split}
\Enp\Bigl(|C(v)| \sum_{w \not\in C(v)} |C(w)|\;\Bigm|\;C(v)\Bigr)
& = |C(v)| \cdot S_{n-|C(v)|}(p) \le  |C(v)| \cdot S_n(p).
\end{split}
\end{equation}
Taking the expectation and summing over all vertices~$v \in [n]$,
we obtain, recalling~\eqref{eq:Q:sum} and \eqref{def:S}--\eqref{def:Sn},
\begin{equation}\label{logder:upper:cond2}
\begin{split}
\E_{n,p}\Bigpar{\sum_{i \neq j}|C_i|^2 |C_j|^2} =
\Enp\Bigl(\sum_{v\in[n]} |C(v)| \sum_{w \not\in C(v)} |C(w)|\Bigr)
\le \Enp \Bigpar{\sum_{v\in[n]} |C(v)|} \cdot S_n(p)= S_n(p)^2,
\end{split}
\end{equation}
which is
\eqref{eq:sum:upper}. 

For~\eqref{eq:sum:upper:cross} we rely on the classical
tree--graph inequalities~\cite[(5.3)--(5.4)]{AizenmanNewman1984} of Aizenman 
and Newman from~1984 (see also~\cite[(6.85)--(6.96)]{Grimmett1999} for a 
modern exposition). As noted in~\cite[p.~123]{AizenmanNewman1984}, their 
proofs apply directly to percolation on any finite transitive graph. 
For any integer $k \ge 1$ and vertex~$v \in [n]$, these inequalities 
state (in our notation) that 
\begin{equation}\label{eq:treegraph} 
\E_{n,p}\bigl(|C(v)|^k\bigr) \le (2k-3)!! \cdot \bigl(\E_{n,p}|C(v)|\bigr)^{2k-1} .
\end{equation}
Recalling $S_{n}(p)=n\E_{n,p}|C(v)|$, see~\eqref{Sn-chi}, by 
summing~\eqref{eq:treegraph} with $k=3$ over all vertices~$v \in [n]$ 
we infer 
\[
\E_{n,p}\Bigpar{\sum_{i}|C_i|^4} = \sum_{v \in [n]}\E_{n,p}\bigl(|C(v)|^3\bigr) \le 3 \cdot \bigl[n^{-1}S_n(p)\bigr]^{5} \cdot n ,
\]
which together with~\eqref{eq:sum:upper}
establishes~\eqref{eq:sum:upper:cross}.
\end{proof}
\begin{proof}[Proof of~\eqref{rough:upper} of Theorem~\ref{thm:rough}]
We shall distinguish five (somewhat overlapping) ranges of~$np$ that will be
treated separately.
We begin by noting that~\eqref{eq:SQ:sum} and~\eqref{eq:sum:upper} together imply
\begin{equation}\label{eq:logder:upper}
\frac{\dx}{\dx p} \log S_n(p) \le \frac{S_n(p)}{1-p} ,
\end{equation}
which will be useful in the subcritical and critical cases.

Let $\pi_0$ and $b$ be as in  Corollary \ref{C:largep} and
pick $A \ge \max\{\pi_0,2\}$ such that $exe^{-x/2} \le 1/2$ for $x \ge A$. 
Let $a,B,c>0$ and $\delta \in (0,1/2]$ be the 
constants given in Theorem~\ref{thm:RG}\ref{T:RGb}--\ref{T:RGc}. 
We set $\Lambda := \max\{(B/\delta^2)^{1/3},1\}$, and 
henceforth assume that $n$ is large enough whenever necessary. 
(This is no loss of generality since 
\eqref{def:S}--\eqref{def:Sn} 
and~\eqref{eq:S:der} imply $S_n(p) \ge n$ 
and $\frac{\dx}{\dx p}S_n(p) \le n^4$ 
while $\min\{|p-1/n|^{-1},n^{4/3}\} \ge 1$
for every $n \ge 1$, 
and thus
\eqref{rough:upper} trivially holds for
any fixed $n$ if $C$ is large enough.)

\pfcase{$np = 1-\eps$ with $\eps^3 n \ge 1$.} 
By~\eqref{eq:S-}, 
\begin{equation}\label{mk}
S_{n}(p) \le \eps^{-1}n = |p-1/n|^{-1}.
\end{equation}
Since~$p \le 1/2$ (for $n \ge 2$) 
and $|p-1/n|^{-1} \le n^{4/3}$, 
now~\eqref{eq:logder:upper} and \eqref{mk}
imply \eqref{rough:upper}. 

\pfcase{$|np-1| \le \Lambda n^{-1/3}$.} 
Noting that $np \le 1+\eps$ with $\eps=\Lambda n^{-1/3}$ and using 
the supercritical upper bound~\eqref{eq:S+} for $S_{n}(p)$ it 
follows that 
\[
S_n(p) \le D \eps^2 n^2 = D \Lambda^{2} \cdot n^{4/3} \le D \Lambda^3 \cdot |p-1/n|^{-1} .
\] 
Since~$p \le 1/2$ (for $n \ge 4\Lambda$, say), now~\eqref{eq:logder:upper} 
implies~\eqref{rough:upper}. 

\pfcase{$np = 1+\eps$ with $\Lambda n^{-1/3} \le \eps \le A$.} 
This is a more difficult range. We shall be guided by the
so-called 'symmetry rule', which intuitively states the following: 
after removing the largest component from the supercritical random graph 
$G_{n,p}$ with $np=1+\eps$, the remaining graph resembles a subcritical 
random graph $G_{n',p}$ with suitable $n'$ and $n'p=1-\eps'$, 
see \cite[Section 5.6]{JLR}. 

Let 
\begin{equation}
  \label{gaeps}
\alpha(\eps) = (1-\delta)\rho(\eps) ,
\end{equation} 
so that $(n-\alpha(\eps)n) \cdot p \le 1-c\eps$ 
by~\eqref{C1:supcr:alpha}. Using the subcritical 
estimate~\eqref{eq:S-} of Theorem~\ref{thm:RG}, 
it follows that for $D_1:=c^{-1}$ we have  
\begin{equation}\label{eq:S:duality}
S_{\floor{n-\alpha(\eps) n}}(p) \le (c \eps)^{-1} n = D_1 \eps^{-1} n .
\end{equation}

Note that $|C_1| \le \alpha(\eps) n$ is a decreasing event, and 
that $S(G_{n,p}) = \sum_i|C_i|^2$ and thus $S(G_{n,p})^2 =
\sum_{i,j}|C_i|^2|C_j|^2$ 
are increasing functions of the edge indicators. 
By Harris's inequality (a special case of the FKG-inequality), 
it follows that 
\begin{equation}\label{rough:11}
\begin{split}
\E\Bigpar{\indic{|C_1| \le \alpha(\eps) n}\sum_{i,j} |C_i|^2|C_j|^2}
&\le \Pr\bigpar{|C_1| \le \alpha(\eps)n} \cdot 
\E\Bigpar{\sum_{i,j} |C_i|^2|C_j|^2}.
\end{split}
\end{equation}
Combining \eqref{rough:11} with
\eqref{gaeps}, the tail
estimate~\eqref{C1:supcr:tail} and the
inequality~\eqref{eq:sum:upper:cross}, using the 
upper bound~\eqref{eq:S+} for $S_{n}(p)$, it follows that 
\begin{equation}\label{rough:12}
\begin{split}
\E\Bigpar{\indic{|C_1| \le \alpha(\eps) n}\sum_{i,j} |C_i|^2|C_j|^2}
&\le e^{-a \delta^2\eps^3n} \cdot \Bigl(D \eps^2 n^2 + 3\bigl[D \eps^2n\bigr]^4\Bigr) \cdot S_{n}(p) = O(\eps^{-1}n) \cdot S_{n}(p), 
\end{split}
\end{equation}
where we used $e^{-x} (x+x^3) \le 2$ for the last inequality (and that
$a,\delta,D$ are constants).

Conditioning on (the vertex set of) the largest component~$C_1$ of
$G_{n,p}$, the 
remaining graph with vertex set $[n]\setminus C_1$ has the same distribution as 
$G_{n-|C_1|,p}$ conditioned on the event $\cD_{C_1}$ that all components have 
size at most~$|C_1|$ and that there is no component of size exactly $|C_1|$
with a smaller vertex label than $C_1$. 
Similarly to~\eqref{logder:upper:cond}, it follows that
\begin{equation}\label{rough:13}
\E\Bigl(|C_1|^2 \sum_{i \ge 2} |C_i|^2\;\Bigm|\;C_1\Bigr) 
=|C_1|^2 \cdot 
\E\bigl(S(G_{n-|C_1|,p}) \;\bigm|\; C_1,\, \cD_{C_1}\bigr).
\end{equation}
For any given $C_1$, $\cD_{C_1}$ is a decreasing event 
for the random graph $G_{n-|C_1|,p}$, while $S(G_{n-|C_1|,p})$ 
is an increasing function.
Hence, as in \eqref{rough:11}, 
by Harris's inequality, 
it follows that
\begin{equation}  \label{rough:13.5}
  \E\bigl(S(G_{n-|C_1|,p}) \bigm| C_1,\, \cD_{C_1}\bigr)
\le \E\bigl(S(G_{n-|C_1|,p}) \bigm| C_1\bigr)
=S_{n-|C_1|}(p).
\end{equation}
By \eqref{rough:13}--\eqref{rough:13.5} and
the monotonicity of Lemma~\ref{lem:Sn} 
together with  \eqref{eq:S:duality}, 
we infer 
\begin{equation}\label{rough:14.1}
\begin{split}
\Enp\Bigpar{\indic{|C_1| \ge \alpha(\eps) n}|C_1|^2 \sum_{i \ge 2} |C_i|^2
\Bigm| C_1}  
& \le 
\indic{|C_1| \ge \alpha(\eps) n}|C_1|^2 S_{n-|C_1|}(p) \\
& \le S_{\floor{n-\alpha(\eps) n}}(p)\cdot |C_1|^2
\le D_1 \eps^{-1}n|C_1|^2
\end{split}
\end{equation}
and thus, by taking the expectation and using
\eqref{def:Sn},
\begin{equation}\label{rough:14.2}
\Enp\Bigpar{\indic{|C_1| \ge \alpha(\eps) n}|C_1|^2 \sum_{i \ge 2} |C_i|^2}  
\le D_1 \eps^{-1}n \Enp|C_1|^2
\le D_1 \eps^{-1}n S_n(p).
\end{equation}

Similarly to~\eqref{rough:13}--\eqref{rough:14.2}, 
by combining~\eqref{eq:sum:upper:cross} with the upper 
bound~\eqref{eq:S:duality} for $S_{\floor{n-\alpha(\eps) n}}(p)$, we deduce 
\begin{equation}\label{rough:2}
\begin{split}
\E\Bigpar{\indic{|C_1| \ge \alpha(\eps) n} \sum_{i,j \ge 2}|C_i|^2|C_j|^2}
& \le \E_{\floor{n-\alpha(\eps) n},p}\Bigpar{\sum_{i,j}|C_i|^2|C_j|^2} \\
& \le \Bigl(D_1 \eps^{-1}n + 3 \bigl[D_1 \eps^{-1}\bigr]^4\Bigr) 
\cdot S_{\floor{n-\alpha(\eps) n}}(p) = O(\eps^{-1}n) \cdot  S_{n}(p) ,
\end{split}
\end{equation}
where we used $\eps^3 n \ge\Lambda^3 \ge 1$ and 
$S_{\floor{n-\alpha(\eps) n}}(p) \le S_{n}(p)$ (see Lemma~\ref{lem:Sn}) 
for the final inequality. 

In view of~\eqref{eq:SQ:sum}, using $p \le 1/2$ (for $n \ge 2(1+A)$, say) our 
estimates \eqref{rough:12}, \eqref{rough:14.2} and \eqref{rough:2} imply 
\begin{equation*}
  \frac{\dx}{\dx p}S_n(p) = O(\eps^{-1} n) \cdot S_n(p),
\end{equation*}
which due to $\eps^{-1} n = |p-1/n|^{-1}$
and $|p-1/n|^{-1} \le n^{4/3}/\Lambda$ 
yields \eqref{rough:upper} in this case too.  

\pfcase{$A \le np \le n/2$.}  
In this range many technicalities from the previous case simplify.  
By distinguishing the events $|C_1| \le n/2$ 
and $n/2 < |C_1| \le n$ (in which case $|C_2| \le n - |C_1| < n/2$), 
using $\sum_i |C_i|=n$ we infer
\begin{equation}\label{tripp}
\sum_{i \neq j}|C_i|^2|C_j|^2 
\le \indic{|C_1| \le n/2} n^4 + 2 n^2 \sum_{i}\indic{|C_i| \le n/2}|C_i|^2 .
\end{equation}
As $enpe^{-np/2} \le 1/2$ by the choice of $A$, 
standard component counting arguments 
from random graph theory and Stirling's formula ($k! \ge \sqrt{2 \pi k} (k/e)^k$) yield 
\begin{equation}\label{trapp}
\begin{split}
\E \Bigl( \sum_{i}\indic{|C_i| \le n/2}|C_i|^2 \Bigr) & \le \sum_{1 \le k \le n/2} k^2 \cdot \binom{n}{k}k^{k-2}p^{k-1}(1-p)^{k(n-k)} \\
& \le \sum_{k \ge 1}\frac{(knp)^k e^{-knp/2}}{k!\,p }  
\le \frac{1}p\sum_{k \ge 1}\frac{\bigl(enp e^{-np/2}\bigr)^k}{\sqrt{2\pi k}} 
\le \frac{1}{p} .
\end{split}
\end{equation}
Since $np \ge \pi_0$, by Corollary \ref{C:largep} we see that 
for large $n$ we also have
\begin{equation}\label{trull}
\E( \indic{|C_1| \le n/2} n^4) = \Pr(|C_1| \le n/2) \cdot n^4 
\le n^4e^{- b n}=O(1).
\end{equation}
Inserting \eqref{tripp}--\eqref{trull} into~\eqref{eq:SQ:sum}
and
using $(1-p)^{-1} \le 2$ we obtain 
\begin{equation}
\frac{\dx}{\dx p} S_n(p) = O\bigpar{n^2 p^{-1}}.  
\end{equation}
Since $S_n(p) \ge n^2/8$ by Corollary \ref{C:largep},
this yields $\frac{\dx}{\dx p} \log S_n(p) = O( p^{-1})$,
which establishes~\eqref{rough:upper}
because now $p^{-1} \le |p-1/n|^{-1}$ and $p^{-1} = O(n)=O(n^{4/3})$. 

\pfcase{$(\log n)^2 \le np < n$.} 
This is a less interesting range since with very high probability, $G_{n,p}$
is connected and thus $\sum_i|C_i|^2=|C_1|^2=n^2$. To obtain rigorous estimates,
let $\cE$ denote the monotone increasing event that $G_{n,p}$ is 
$2$-edge connected (after deleting any edge the resulting 
graph remains connected). It is well-known that 
$\Pr_{n,2(\log n)/n}(\neg\cE)=o(1)$ holds (see, e.g.,~\cite{ER1961con}), 
so a multi-round exposure argument yields 
$\Pr_{n,p}(\neg\cE) \le 
\Pr_{n,2\log n/n}(\neg\cE)^{\floor{np/2\log n}} \le 
n^{-\omega(1)}$.  
Observe that if $\cE$ holds, then no edge can be pivotal for 
the event $v \leftrightarrow w$. Using~\eqref{eq:S:der} we infer 
\[
\frac{\dx}{\dx p} S_n(p) \le n^4 \cdot \Pr_{n,p}(\neg\cE) \le n^{-\omega(1)},
\] 
which together with $S_n(p) \ge 1$ and $|p-1/n|\ge 1/n$ 
completes the proof of~\eqref{rough:upper}.  
\end{proof}

\subsection{Lower bound}\label{sec:maxder:crit}
In this subsection we focus on the lower bound~\eqref{rough:lower} in
Theorem~\ref{thm:rough}. 
Our proof strategy is to consider the event that $G_{n,p}$ contains two
distinct components of size $\Theta(n^{2/3})$.  
\begin{lemma}\label{Lch}
  Let $\cL$ be the event that $|C_2|\ge n^{2/3}$, i.e., 
that $G_{n,p}$ contains two distinct components
with at least $n^{2/3}$ vertices each.
For every $\lambda\in\mathbb R$ there  
exist constants $\delta_\lambda,n_0>0$ such that, for all $n \ge n_0$,
if  $p=1/n+ \lambda n^{-4/3}$, then
\begin{equation}
  \label{magnus}
\Pr_{n,p}(\cL)\ge\gd_\lambda.
\end{equation}
\end{lemma}
\begin{proof}
This follows immediately from \cite[Corollary 2]{Aldous1997}. 
(See also \cite[Theorem~5.20]{JLR}, there stated for $G(n,m)$.)
\end{proof}

\begin{proof}[Proof of~\eqref{rough:lower} of Theorem~\ref{thm:rough}]
As for the upper bound, we may assume that $n$ is large enough, since 
\eqref{rough:lower} trivially holds (if $D_\lambda$ is chosen small enough)
for every fixed $n\ge2$ because
$S_n(p)$ and $\frac{\dx}{\dx p}S_n(p)$ are positive functions on
$(0,1)$.

With $\cL$ as in Lemma \ref{Lch}, we have
\[
\sum_{i \neq j}|C_i|^2 |C_j|^2  \ge \indic{\cL} \bigl(n^{2/3}\bigr)^4,
\]
and thus by 
\eqref{eq:S:der:Q}--\eqref{eq:Q:sum} and \eqref{magnus}
\begin{equation}
\frac{\dx}{\dx p} S_n(p)
\ge \E_{n,p}\Bigpar{\sum_{i \neq j}|C_i|^2 |C_j|^2}
\ge \Pr_{n,p}(\cL) n^{8/3}
 \ge \delta_\lambda \cdot n^{8/3}  .
\end{equation}
By \eqref{eq:S+} 
(with $\eps:=\max\set{1,\gl}n^{-1/3}$)
we also know that 
$p=1/n+ \lambda n^{-4/3}$ implies $S_n(p) \le C_{\lambda} n^{4/3}$,
establishing~\eqref{rough:lower} with~$D_\lambda = \delta_{\lambda}/C_\lambda$ 
for $n$ sufficiently large.
\end{proof}

\begin{remark} 
  Although we have stated \eqref{rough:lower} and Lemma \ref{Lch} for a
  fixed $\gl$, the results hold uniformly for $\gl$ in any compact interval,
  i.e., we can take $D_\gl$ and $\delta_\gl$ independent of
  $\gl\in[-\Lambda,\Lambda]$ for any $\Lambda>0$, provided we assume for
  example $n\ge2\max\set{1,\Lambda^3}$ (to guarantee that $p\in(0,1)$).
This follows from the more refined Theorem \ref{thm:main2}, but it can also
be seen from the simple proof above by noting that
the result in \cite[Corollary~2]{Aldous1997}, although stated for 
$p=1/n+ \lambda n^{-4/3}$ for a fixed $\lambda$, also holds (by the same
proof)
more generally for
$p=1/n+ (\lambda+o(1)) n^{-4/3}$;
it then follows from Lemma \ref{Luc} below that 
for $\gl\in[-\Lambda,\Lambda]$,
$\Pr_{n,n^{-1}+\gl n^{-4/3}}(\cL)$ converges uniformly to a continuous
positive function, which
yields a uniform lower bound in \eqref{magnus}, and thus in \eqref{rough:lower}.
\end{remark}

\section{Scaling inside the critical window}\label{sec:scaling}
In this section we prove Theorem~\ref{thm:main2}.
Our arguments exploit that inside the critical window, the rescaled sizes of
the largest components converge to some random variables (as mentioned in
the introduction).

Following \cite{JansonSpencer2007}, we 
define
\begin{align}\label{def:Lambda}
\Lambda^{(\lambda)}(x) &:= (2\pi)^{-1/2} x^{-5/2}e^{-F(x,\lambda)}\sum_{\ell
  \ge 0} w_{\ell} x^{3\ell/2} , 
\end{align}
where
\begin{align}
\label{def:Fxl}
F(x,\lambda) &:= \bigl((x-\lambda)^3+\lambda^3\bigr)/6 ,
\end{align}
and $w_\ell$, $\ell\ge0$, are Wright's constants
\cite{Wright1977},
which as shown by Spencer \cite{Spencer1997}  can be expressed as 
\begin{equation}
  \label{def:wl}
w_{\ell}= \E(\Bex^{\ell})/\ell!, 
\end{equation}
where the random variable~$\Bex$ is the area under a normalized Brownian
excursion, see also the survey~\cite{Janson2007}. 
As shown in \cite[Theorem~4.1]{JansonSpencer2007},
$\Lambda^{(\lambda)}(x)$ is the intensity of the point process that 
by \cite{Aldous1997} describes asymptotically the sequence 
$(|C_i|/n^{2/3})_{i\ge1}$, and we define the
corresponding moments
\begin{equation}\label{def:fk}
f_k(\lambda) = \int_{0}^{\infty} x^k \Lambda^{(\lambda)}(x) \dx x,
\qquad k\ge2.
\end{equation}
As remarked in 
\cite[after Corollary~4.2]{JansonSpencer2007},
$\Lambda^{(\lambda)}(x)$ decreases exponentially as $x\tooo$, and is $\Theta
\bigpar{x^{-5/2}}$ as $x\to0$; hence the integral \eqref{def:fk} converges
so $0<f_k(\lambda)<\infty$ for every $k\ge2$ and $\lambda\in\mathbb R$.

By~\eqref{def:Fxl} we have  $\frac{\partial}{\partial\lambda} F(x,\lambda) = -x^2/2 +
\lambda x$ and thus by \eqref{def:Lambda}
$\frac{\partial}{\partial\gl}\gLL(x)=\bigpar{\frac{x^2}2-\gl  x}\gLL(x)$.
Hence, by differentiating inside the integral in \eqref{def:fk} (which is
easily justified, e.g.\ using dominated convergence),
$f_k(\gl)$ is differentiable and
\begin{equation}\label{eq:fk:der}
\frac{\dx}{\dx \lambda} f_k(\lambda) 
= \int_{0}^{\infty} x^k \frac{\partial \Lambda^{(\lambda)}(x)}{\partial \gl} \dx x
= \int_{0}^{\infty} x^k \Bigpar{\frac{x^2}2-\gl x}\Lambda^{(\lambda)}(x) \dx x
= \frac12 f_{k+2}(\lambda) - \lambda f_{k+1}(\lambda) .
\end{equation}
By induction, $f_k(\gl)$ is infinitely differentiable for every $k\ge2$.

Recall now \eqref{def:Sn}, and note that \eqref{eq:SQ:sum} can be written as
\begin{equation}\label{SnCiDer}
\frac{\dx}{\dx p}S_n(p) 
= \frac{\E_{n,p} \bigpar{\bigl(\sum_i |C_i|^2\bigr)^2}}{1-p} 
- \frac{\E_{n,p} \bigpar{\sum_i |C_i|^4}}{1-p}.
\end{equation}
To treat such sums, we first
note the following fact, which is stated in \cite[Theorem~B1 and 
Remark~B2]{JansonLuczak2008} as an immediate consequence of 
results of Aldous~\cite{Aldous1997} and Janson and
Spencer~\cite{JansonSpencer2007}. 
\begin{lemma}[\cite{Aldous1997,JansonSpencer2007,JansonLuczak2008}]\label{lem:conv}
Let $\lambda \in \RR$ and $k\in \NN$ with $k \ge 2$.
Then there exists a random variable $W_{\lambda,k}$ 
with 
\begin{equation}\label{eq:conv:exp}
\E W_{\lambda,k} = f_k(\lambda) ,
\end{equation}
such that for $p = 1/n+(\lambda + o(1))n^{-4/3}$ we have 
\begin{equation}
\label{eq:conv:sum}
\frac{\sum_{i}|C_i|^k}{n^{2k/3}} \dto W_{\lambda,k}.
\end{equation}
\qed
\end{lemma}
In Section~\ref{sec:convergence} we justify taking expectations, higher
moments and 
derivatives in~\eqref{eq:conv:sum}, and use this to establish the
convergence results~\eqref{main:conv}--\eqref{main:der} of
Theorem~\ref{thm:main2} with $f=f_2$.
(In Appendix~\ref{AppAhigher} we extend this argument to 
higher derivatives.) 
Finally, in Section~\ref{sec:zero} we complete the proof of
Theorem~\ref{thm:main2} by showing 
$\frac{\dx \log f_2}{\dx \lambda}(0) \neq 0$  
numerically via a series expansion (that converges exponentially).

\subsection{Convergence}\label{sec:convergence}
In this subsection we prove the convergence results~\eqref{main:conv}--\eqref{main:der} of Theorem~\ref{thm:main2}
using the distributional convergence~\eqref{eq:conv:sum} from
Lemma~\ref{lem:conv} and the following auxiliary result. 
\begin{theorem}%
\label{thm:scaling}
Let $D,\lambda \in \RR$ and $k,q \in \NN$ with $k \ge 2$ and $q \ge 1$. 
\begin{romenumerate}
\item \label{thm:scaling-i}
There exists  $C=C(D,k,q)$ such that, for all $n \ge 1$ and $p \in (0,1)$ 
satisfying $p \le 1/n + D n^{-4/3}$,  
\begin{equation}\label{bound:sum:sq}
0 \le \E_{n,p}\lrpar{ \biggl(\frac{\sum_{i}|C_i|^k}{n^{2k/3}}\biggr)^q}\le C .
\end{equation}
\item \label{thm:scaling-ii}
For $p = 1/n + (\lambda + o(1)) n^{-4/3}$ we have 
\begin{equation}\label{eq:conv:sum:sq}
\E_{n,p}\lrpar{\biggl(\frac{\sum_{i}|C_i|^k}{n^{2k/3}}\biggr)^q} 
\: \xrightarrow{n\to\infty} \: \E \bigpar{\lrpar{W_{\lambda,k}}^q}<\infty , 
\end{equation}
\end{romenumerate}
where the random variable $W_{\lambda,k}$ is defined as in Lemma~\ref{lem:conv}.
Moreover, the limit in \eqref{eq:conv:sum:sq} is a continuous function of
$\lambda$, and 
if $p=1/n+\lambda n^{-4/3}$, then
the convergence in \eqref{eq:conv:sum:sq} is uniform for 
$\lambda$ in any compact interval $[\lambda_1,\lambda_2]\subset\RR$.
\end{theorem}
\begin{proof}
We start with the uniform moment bound~\eqref{bound:sum:sq}.
Since $\sum_i|C_i|^k$ does not decrease if any edge is added, 
the expectation is a monotone function of $p$; thus it suffices to consider
 $p = 1/n + D n^{-4/3}$. 
As a warm-up, we first consider the special case $q=2$. 
Similarly to~\eqref{eq:Q:sum} we have 
\[
\Bigpar{\sum_i |C_i|^k}^{\raisebox{-2pt}{$\scriptstyle2$}} = \sum_{v \in [n]} |C(v)|^{k-1} \sum_{w \notin C(v)} |C(w)|^{k-1} + \sum_{v \in [n]}|C(v)|^{2k-1} .
\]
Mimicking the conditioning and monotonicity arguments 
leading to~\eqref{eq:sum:upper}, see~\eqref{logder:upper:cond},
we infer that 
\begin{equation*}\label{bound:sum:sq:2}
\E_{n,p} \biggpar{\biggl(\frac{\sum_{i}|C_i|^k}{n^{2k/3}}\biggr)^2} 
\le \frac{\bigl(\E_{n,p} \sum_{v \in [n]} |C(v)|^{k-1}\bigr)^2}{n^{4k/3}} + \frac{\E_{n,p} \sum_{v \in [n]}|C(v)|^{2k-1}}{n^{4k/3}} .
\end{equation*}
Generalizing the above argument, for every integer $q \ge 1$ there is a constant $A_{k,q}$ such that 
\begin{equation}\label{bound:sum:sq:q}
\E_{n,p} \biggpar{\biggl(\frac{\sum_{i}|C_i|^k}{n^{2k/3}}\biggr)^q} 
\le A_{k,q} \sum_{1 \le r \le q} \sum_{\substack{j_1 + \cdots j_r = qk:\\ k | j_i}} \frac{\prod_{1 \le i \le r} \E_{n,p} \sum_{v \in [n]} |C(v)|^{j_i-1}}{n^{2qk/3}} . 
\end{equation}
By \cite[Corollary~5.3]{JansonSpencer2007} 
(or by inserting $\E_{n,p}|C(v)|=n^{-1}S_n(p)=O(n^{1/3})$, see~\eqref{Sn-chi} and~\eqref{eq:S+}, into~\eqref{eq:treegraph})  
there are constants $(B_{j,D})_{j\ge 2}$ such that
\begin{equation}
  \label{jul}
 \E_{n,p} \sum_{v \in [n]} |C(v)|^{j-1} 
=
 \E_{n,p} \sum_{i} |C_i|^{j}
\le B_{j,D} n^{2j/3}.
\end{equation}
Since $j _i \ge k \ge 2$, \eqref{jul} applies to each 
factor in each product in~\eqref{bound:sum:sq:q}, 
so~\eqref{bound:sum:sq} follows for suitable~$C=C(D,k,q)$.

We next turn to~\ref{thm:scaling-ii},
and thus assume $p=1/n + (\lambda + o(1)) n^{-4/3}$. 
For brevity we write
\begin{equation}\label{def:Xnkq}
X_{n,k} := \frac{\sum_{i}|C_i|^k}{n^{2k/3}}.
\end{equation}
The upper bound~\eqref{bound:sum:sq}, with $2q$, say, shows that the
random variables $\bigl((X_{n,k})^q\bigr)_{n \ge 1}$ are uniformly integrable for fixed
$k \ge 2$ and $q \ge 1$,
see e.g.~\cite[Theorem~5.4.2]{Gut2005}. 
Since $X_{n,k} \dto W_{\lambda,k}$ 
by~\eqref{eq:conv:sum}, and thus $(X_{n,k})^q \dto (W_{\lambda,k})^q$ 
by the continuous mapping theorem~\cite[Theorem~5.10.4]{Gut2005}),
it thus follows that $\E_{n,p}((X_{n,k})^q) \to \E ((W_{\lambda,k})^q)<\infty$
as $n \to \infty$, see~\cite[Theorem~5.5.9]{Gut2005}, which
completes the proof of~\eqref{eq:conv:sum:sq}.  

The final claims now follow by the following elementary calculus lemma.
\end{proof}

\begin{lemma}\label{Luc}
Suppose that $h_n(\lambda)$ and $h(\lambda)$ are real-valued functions on $\RR$
  such that if $\gl\in\RR$ and $\gl_n=\gl+o(1)$, then $h_n(\gl_n)\to
  h(\gl)$ as $n\tooo$.
Then $h(\gl)$ is continuous and
$h_n(\gl)\to h(\gl)$ uniformly for $\gl$ in any compact set.
\end{lemma}
\begin{proof}
  First, suppose that $h$ is discontinuous at some $\gl$.
Then there exist $\eps>0$ and a sequence $\gl_k\to\gl$ such that
$|h(\gl_k)-h(\gl)|>\eps$ for all $k$.
Since $h_n(\gl_k)\to h(\gl_k)$,
we may find an increasing sequence $n_k$ such that
$|h_{n_k}(\gl_k)-h(\gl_k)|<\eps/2$. Then $|h_{n_k}(\gl_k)-h(\gl)|>\eps/2$.
On the other hand, the assumption implies $h_{n_k}(\gl_k)\to h(\gl)$, a
contradiction. 

Similarly, assume that $h_n(\gl)$ does not converge uniformly to $h(\gl)$ on
  the compact set $K$. Then there exist $\eps>0$ and sequences
$n_k\to\infty$ and $\gl_k\in K$ such that
  $|h_{n_k}(\gl_k)-h(\gl_k)|>\eps$.
Since $K$ is compact, we may select a subsequence such that, along
this subsequence,
$\gl_k\to\gl$ for some $\gl$.
Then the assumption and the continuity of $h$ just shown imply that, along
the subsequence, 
$h_{n_k}(\gl_k)\to h(\gl)$ and $h(\gl_k)\to h(\gl)$, and thus
$h_{n_k}(\gl_k)-h(\gl_k)\to0$, a contradiction.
\end{proof}
\begin{remark}\label{Rluc}
Lemma~\ref{Luc} is valid for functions on any metric space. (Also the ranges
of the functions may be in an arbitrary metric space.)
Furthermore, 
  the converse of the lemma also holds (and is easy): 
if $h_n(\gl)\to h(\gl)$ uniformly on compact sets
  and $h$ is continuous, then $h_n(\gl_n)\to h(\gl)$ whenever $\gl_n\to \gl$.
\end{remark}

\begin{proof}[Proof of~\eqref{main:conv}--\eqref{main:der} of Theorem~\ref{thm:main2}]
Recall that $\goodchi_{K_n}(p) = S_n(p)/n$ by \eqref{Sn-chi}.  
Define~$X_{n,k}$ as in~\eqref{def:Xnkq}. 
Using also~\eqref{def:Sn}, \eqref{eq:conv:sum:sq}  and \eqref{eq:conv:exp},
we have 
\begin{equation}\label{SnCi:conv:2}
\frac{\goodchi_{K_n}(p)}{n^{1/3}} =\frac{S_n(p)}{n^{4/3}} = \E_{n,p} X_{n,2} \: \xrightarrow{n\to\infty} \: \E W_{\lambda,2} = f_2(\lambda) > 0 ,
\end{equation}
where $f_2(\lambda)>0$ follows from the definition \eqref{def:fk}.
This proves \eqref{main:conv}.

Similarly, by~\eqref{SnCiDer}, \eqref{eq:conv:sum:sq}  and
\eqref{eq:conv:exp}, we have 
\begin{equation}\label{SnCi:conv:der}
\frac{\frac{\dx}{\dx p}\goodchi_{K_n}(p)}{n^{5/3}}
 = \frac{\frac{\dx}{\dx p}S_n(p)}{n^{8/3}} 
= \frac{\E_{n,p} (X_{n,2}^2)}{1-p} - \frac{\E_{n,p} X_{n,4}}{1-p} \: \xrightarrow{n\to\infty} \: \E (W_{\lambda,2}^2) - \E W_{\lambda,4} =: g(\lambda) .
\end{equation}

By Lemma \ref{Luc}, $f_2(\gl)$ and $g(\gl)$ are continuous, and for
$p=1/n+\gl n^{-4/3}$, the limits hold uniformly on compact sets. 
Combining~\eqref{SnCi:conv:2}--\eqref{SnCi:conv:der} we infer
\[
\frac{\frac{\dx}{\dx p} \log \goodchi_{K_n}(p)}{n^{4/3}} \: \xrightarrow{n\to\infty} \: \frac{g(\lambda)}{f_2(\lambda)},
\] 
and thus $g(\lambda)/f_2(\lambda) > 0$ follows from~\eqref{eq:logder:chi:S}
and~\eqref{rough:lower}. 

It remains to prove that $g(\lambda) = \frac{\dx}{\dx \lambda} f_2(\lambda)$ holds. 
To this end we fix $\lambda_1,\lambda_2 \in \RR$ with $\lambda_1 < \lambda_2$, and set $p_i = 1/n + \lambda_i n^{-4/3}$. 
By~\eqref{SnCi:conv:2} we have  
\begin{equation}
\label{eq:scaling:diff:f}
\int_{p_1}^{p_2} \frac{\frac{\dx}{\dx p}  \goodchi_{K_n}(p)}{n^{1/3}} \dd p = \frac{\goodchi_{K_n}(p_2)}{n^{1/3}} - \frac{\goodchi_{K_n}(p_1)}{n^{1/3}} \: \xrightarrow{n\to\infty} \: f_2(\lambda_2)-f_2(\lambda_1) .
\end{equation}
On the other hand, by substituting $p = 1/n + \lambda n^{-4/3}$ we have by
the uniform convergence in \eqref{SnCi:conv:der} just shown (or by
dominated convergence and~\eqref{bound:sum:sq})
that
\begin{equation}
\label{eq:scaling:diff:intg}
\int_{p_1}^{p_2} \frac{\frac{\dx}{\dx p}  \goodchi_{K_n}(p)}{n^{1/3}} \dd p 
= \int_{\lambda_1}^{\lambda_2} \frac{\frac{\dx}{\dx p}  \goodchi_{K_n}(p) \big|_{p = 1/n + \lambda n^{-4/3}}}{n^{5/3}} \dd\lambda \: \xrightarrow{n\to\infty} \: \int_{\lambda_1}^{\lambda_2} g(\lambda) \dd\lambda .
\end{equation}
It follows from~\eqref{eq:scaling:diff:f}--\eqref{eq:scaling:diff:intg} that 
$f_2(\lambda_2)-f_2(\lambda_1) = \int_{\lambda_1}^{\lambda_2} g(\lambda) \dd\lambda$. 
Since $\lambda_1 < \lambda_2$ were arbitrary, and $g(\gl)$ is continuous,
it follows that $g(\lambda) = \frac{\dx}{\dx \lambda} f_2(\lambda)$ 
for all $\lambda \in \RR$, completing the proof. 
\end{proof}

\subsection{Explicit bounds for $\lambda=0$}\label{sec:zero}
In this subsection we complete the proof of Theorem~\ref{thm:main2}, and by
the arguments of Section~\ref{sec:convergence} it remains to prove the
following technical lemma.
\begin{lemma}\label{lem:conv:zero}
Define the function $f_2$ as in~\eqref{def:fk}. 
Then $\frac{\dx^2}{\dx \lambda^2} \log f_2(0) \neq 0$. 
\end{lemma}
\begin{remark}\label{rem:conv:zero}
The proof of Lemma~\ref{lem:conv:zero} shows that 
$\frac{\dx^2}{\dx \lambda^2} \log f_2(0) \approx 0.296833365232$. 
\end{remark}
The idea is to give (in the special case $\lambda=0$) rigorous numerical 
estimates for the right hand side of 
\begin{equation}\label{eq:fk:der2}
\frac{\dx^2}{\dx \lambda^2}\log f_2(\lambda) = \frac{f_2(\lambda) \frac{\dx^2}{\dx \lambda^2} f_2(\lambda)- \bigl(\frac{\dx}{\dx \gl} f_2(\lambda)\bigr)^2}{\bigl(f_2(\lambda)\bigr)^2} .
\end{equation}
The derivatives can be computed by \eqref{eq:fk:der}.
In the special case $\lambda=0$ we obtain 
\begin{align}
\frac{\dx }{\dx \lambda}f_2(0)& = \tfrac12f_{4}(0) , \\
\frac{\dx^2 }{\dx \lambda^2}f_2(0) &=\tfrac{1}4 f_6(0)-f_3(0) .
\end{align}
Furthermore, by~\cite[Remark~6]{JansonSpencer2007} we also have 
(in our notation) the identity $f_3(\lambda) = 2 + 2 \lambda f_2(\lambda)$, 
so that $f_3(0)=2$.
Hence~\eqref{eq:fk:der2} yields
\begin{equation}\label{eq:fk:der0}
\frac{\dx^2}{\dx \lambda^2} \log f_2(0) = 
\frac{f_2(0) \bigl(f_6(0) - 8\bigr) - f_4(0)^2}{4f_2(0)^2} ,
\end{equation}
and due to $0 < f_2(\lambda) < \infty$ 
our task is reduced to showing that 
\begin{equation}\label{leontes}
f_2(0) f_6(0) - 8 f_2(0) - f_4(0)^2\neq0.  
\end{equation}
To evaluate the terms in~\eqref{leontes}, note that
by Tonelli's theorem,
the function $f_k$ defined in~\eqref{def:fk} can be written~as  
\begin{equation}\label{eq:fk}
f_k(\lambda) = \sum_{\ell \ge 0} (2\pi)^{-1/2} w_{\ell} \int_{0}^{\infty}
x^{k+3\ell/2-5/2}e^{-F(x,\lambda)} \dx x . 
\end{equation}
Since $F(x,0)=x^3/6$, see~\eqref{def:Fxl}, we can in the case $\gl=0$
evaluate the integral in~\eqref{eq:fk} using the gamma function
$\Gamma(z):=\int _{0}^{\infty }x^{z-1}e^{-x} \dx x$. 
We define, using the substitution $y=x^3/6$,
\begin{equation}\label{eq:fk:int:0}
I_{k,\ell} := \int_{0}^{\infty} x^{k+3\ell/2-5/2}e^{-F(x,0)} \dx x 
= \frac{1}3 \cdot 6^{k/3 +(\ell-1)/2} \cdot \Gamma(k/3+(\ell-1)/2) ,
\end{equation}
and by~\eqref{eq:fk} thus have 
\begin{equation}\label{eq:fk0}
f_k(0) = \sum_{\ell \ge 0} (2\pi)^{-1/2} w_{\ell}I_{k,\ell}.
\end{equation}
The plan is to truncate the infinite sum in~\eqref{eq:fk0} with the 
help of the following uniform estimates.  
\begin{lemma}\label{lem:IKLwl}
For all $k \ge 2$ and $\ell \ge 1$ we have $I_{k,\ell},w_{\ell} \ge 0$ and 
\begin{align}
\label{eq:IKL:bounds}
I_{k,\ell} & \le 2 \pi^{1/2} (3\ell)^{k/3-1}
\Bigl(\frac{3\ell}{e}\Bigr)^{\ell/2} \cdot e^{2k^2/9\ell}, \\
\label{eq:wl:bounds}
w_{\ell} & \le 8 \pi^{-1/2} \sqrt{\ell} \Bigl(\frac{e}{12\ell}\Bigr)^{\ell/2} .
\end{align}
\end{lemma}
The upper bound \eqref{eq:wl:bounds} is off from the asymptotic
value in \cite[(52)]{Janson2007} only by a factor $8/3$.
\begin{proof}
The lower bounds $I_{k,\ell} \ge 0$ and $w_{\ell} \ge 0$ are trivial. 
Turning to upper bounds, we start with $I_{k,\ell}$. 
We use the well-known Stirling-type estimate (see,
e.g.,~\cite[(5.6.1)]{NIST}) 
\begin{equation}\label{gamma:stir}
1 \le \frac{\Gamma(m)}{\sqrt{2\pi/m} (m/e)^m} \le e^{1/12m} . 
\end{equation}
Inserting~\eqref{gamma:stir} into~\eqref{eq:fk:int:0}, it follows by a
simple calculation that, using $k\ge2$ and and $1+ x \le e^x$,
\begin{equation*}\label{eq:IKL:bounds:0}
\begin{split}
I_{k,\ell} 
&\le \frac{1}{3} \cdot \frac{\sqrt{2 \pi}}{\sqrt{k/3 + (\ell-1)/2}} \cdot
\Bigl(\frac{2k + 3(\ell-1)}{e}\Bigr)^{k/3 +(\ell-1)/2} \cdot e^{1/4k} \\  
& \le \frac{\sqrt{4\pi} }{3\sqrt{\ell}} \cdot \Bigl(\frac{3\ell}{e}\Bigr)^{k/3 + (\ell-1)/2} \cdot \Bigl(1 + \frac{2k-3}{3\ell}\Bigr)^{k/3 + \ell/2} \cdot e^{1/4k}. 
\\
&\le \frac{\sqrt{4\pi} }{3 \sqrt{\ell}} \cdot \Bigl(\frac{3\ell}{e}\Bigr)^{k/3 + (\ell-1)/2} \cdot e^{1/4k + 2k^2/9\ell + k/3-1/2} \\
& \le \frac{\sqrt{4\pi}}{\sqrt{3}} \cdot (3\ell)^{k/3-1} \cdot \Bigl(\frac{3\ell}{e}\Bigr)^{\ell/2} \cdot e^{1/8 + 2k^2/9\ell},
\end{split}
\end{equation*}
which due to $\sqrt{4 /3}e^{1/8} <2$ completes the proof of~\eqref{eq:IKL:bounds}.

For $w_{\ell}$, we combine \eqref{def:wl}
with a recurrence formula by Louchard~\cite{Louchard1984} for the Brownian
excursion area, see~\cite[(4)~and~(5)]{Janson2007}:
in the recursion \cite[(5)]{Janson2007} all $\gamma_i$ are positive by
\cite[(4)]{Janson2007}, so the first term on the right hand side 
of~\cite[(5)]{Janson2007} gives an upper bound, which implies 
\begin{equation}\label{eq:wl:bounds:0}
w_{\ell}=  \frac{1}{\ell!} \cdot\E(\Bex^\ell)
\le \frac{1}{\ell!} \cdot \frac{2 \sqrt{\pi}}{(36 \sqrt{2})^\ell \Gamma((3\ell-1)/2)} \cdot \frac{12 \ell}{6\ell-1} \frac{\Gamma(3\ell+1/2)}{\Gamma(\ell+1/2)} .
\end{equation}
Using $\ell!=\ell\Gamma(\ell)$ and the estimate~\eqref{gamma:stir} four
times, it
follows by a simple (but slightly tedious) calculation that 
\begin{equation*}
\begin{split}
w_{\ell} &\le \frac{1}{\sqrt{\pi \ell}}\Bigl(\frac{e}{36
  \sqrt{2}\ell}\Bigr)^\ell \cdot \frac{3\ell-1}{2 \sqrt{e}}
\Bigl(\frac{2e}{3\ell-1}\Bigr)^{3\ell/2} \cdot 
\frac{12\ell}{6\ell-1} \Bigl(\frac{3\ell
  +1/2}{e}\Bigr)^{3\ell} \Bigl(\frac{e}{\ell+1/2}\Bigr)^{\ell} e^{1/36\ell}
\\ 
& \le \frac{3 e^{1/36\ell}}{\sqrt{\pi e}} \cdot \sqrt{\ell} \cdot
\Bigl(\frac{e^2}{36 \sqrt{2}\ell^2}\Bigr)^\ell \cdot 2^{3\ell/2}\Bigl(1 +
\frac{3/2}{3\ell-1}\Bigr)^{3\ell/2} \cdot
\Bigl(\frac{3\ell}{e}\Bigr)^{3\ell/2}\Bigl(1 +
\frac{1}{6\ell}\Bigr)^{3\ell/2} 
\\
& \le \frac{3 e^{(1/36 + 9/8 + 3/12-1/2)}}{\sqrt{\pi }} \cdot \sqrt{\ell}
\cdot \Bigl(\frac{e}{12 \ell}\Bigr)^{\ell/2} 
\le 8 \pi^{-1/2} \sqrt{\ell} \Bigl(\frac{e}{12 \ell}\Bigr)^{\ell/2},
\end{split}
\end{equation*}
completing the proof of~\eqref{eq:wl:bounds}. 
\end{proof}
\begin{lemma}\label{Ltailsum}
  For every real $s\ge0$ and integer $\ell_0\ge 2s$,
  \begin{equation*}
	\sum_{\ell>\ell_0} \ell^s2^{-\ell} \le5\ell_0^s2^{-\ell_0}.
  \end{equation*}
  \begin{proof}
	Let $a_\ell:=l^{s}2^{-\ell}$. For $\ell\ge\ell_0$, we have
	\begin{equation*}
	  \frac{a_{\ell+1}}{a_{\ell}}
=\Bigpar{1+\frac{1}{\ell}}^s2^{-1}
\le e^{s/\ell}2^{-1}
\le e^{1/2}2^{-1} < \frac{5}6.
	\end{equation*}
Hence, $a_{\ell}\le (5/6)^{\ell-\ell_0}a_{\ell_0}$
and the result follows by summing a geometric series.
  \end{proof}
\end{lemma}
\begin{corollary}\label{cor:wlIkl}
For all integers $k\ge2$ and $\ell_0 \ge 2k/3-1$, 
\begin{equation}\label{eq:wlIkl:bounds}
0 \le f_k(0) - \sum_{0 \le \ell \le \ell_0} (2\pi)^{-1/2} w_{\ell} I_{k,\ell} 
\le 
11\,e^{2k^2/9\ell_0}3^{k/3} \ell_0^{k/3-1/2}2^{-\ell_0} .
\end{equation}
\end{corollary}
\begin{proof}
The lower bound in~\eqref{eq:wlIkl:bounds} is trivial by~\eqref{eq:fk}, 
\eqref{eq:fk:int:0} and $I_{k,\ell},w_{\ell} \ge 0$. 
Turning to the upper bound, 
\eqref{eq:IKL:bounds}--\eqref{eq:wl:bounds} yield
\[
(2\pi)^{-1/2} w_{\ell} I_{k,\ell} 
\le 
\frac{16}{\sqrt{2\pi}}e^{2k^2/9\ell}3^{k/3-1} \ell^{k/3-1/2} 2^{-\ell} .
\]
Lemma \ref{Ltailsum} thus gives
\begin{equation*}
  \sum_{ \ell> \ell_0} (2\pi)^{-1/2} w_{\ell} I_{k,\ell} 
\le 
\frac{80}{3\sqrt{2\pi}}e^{2k^2/9\ell_0}3^{k/3} \ell_0^{k/3-1/2} 2^{-\ell_0} ,
\end{equation*}
and the result follows.
\end{proof}
The constants $w_\ell$ are easily computed by recursion, 
see, e.g.,~\cite[(4)--(5) or (6)--(7)]{Janson2007}, so the finite sum
$\sum_{0 \le \ell \le \ell_0} (2\pi)^{-1/2} w_{\ell} I_{k,\ell} $ can be
computed numerically (with arbitrary precision) for any $\ell_0$ that is not
too large. Together with the estimate in Corollary~\ref{cor:wlIkl} of the
remainder, which can be made arbitrarily small by choosing a suitable
$\ell_0$, this enables us to compute $f_k(0)$ with arbitrary precision for
any $k \ge 2$.

\begin{proof}[Proof of Lemma~\ref{lem:conv:zero}]
Choosing $\ell_0=75$, the right hand side of~\eqref{eq:wlIkl:bounds} is 
less than $10^{-17}$ for all $2 \le k \le 6$, with room to spare. 
Proceeding as discussed above, we then obtain (using Maple) 
\begin{align}
  f_2(0)&\doteq 1.830470321422761,\label{f20}\\
  f_4(0)&\doteq 3.514851319980978,\\
  f_6(0)&\doteq 16.922562003970612,
\end{align}
where $\doteq$ means equality for all but the last digit 
(which might be off by one). Hence 
\begin{equation}\label{num0}
f_2(0)f_6(0) - 8f_2(0) - f_4(0)^2 \doteq 3.9783051377505, 
\end{equation}
which shows \eqref{leontes} and thus
completes the proof of Lemma~\ref{lem:conv:zero}.
Remark~\ref{rem:conv:zero} follows by inserting 
\eqref{num0} and \eqref{f20} into~\eqref{eq:fk:der0}.  
\end{proof}

\appendix

\section{Asymptotics of $f(\gl)$ as $\gl\to\pm\infty$}\label{AppAgl}
In this appendix we prove the asymptotics of the function $f(\gl)=f_2(\gl)$
stated after Theorem \ref{thm:main2}, and extend the results 
to $f_k(\gl)$ for arbitrary~$k \ge 2$.
\begin{theorem}\label{Tappa}
Define $f_k:\RR \to (0,\infty)$ as in~\eqref{def:fk}.
 For any fixed $k\ge2$,
$f_k(\gl)$ has the asymptotics
  \begin{align}\label{foo-}
	f_k(\gl)&=\frac{(2k-5)!!}{|\gl|^{2k-3}}\bigpar{1+O(|\gl|^{-3})}
&&\text{as } \gl\to-\infty,
\\\label{foo+}
f_k(\gl)&=(2\gl)^k\bigpar{1+o(1)},
&&\text{as } \gl\to+\infty.
  \end{align}
\end{theorem}
Here, $(2k-5)!!$ is the usual semifactorial, 
i.e., $(2k-5)!!=\prod_{j=1}^{k-2}(2j-1)$,
with $(-1)!!=1$. 
In particular, for $f=f_2$, we have
$f(\gl)\sim|\gl|^{-1}$ as $\gl\to-\infty$ and
$f(\gl)\sim 4\gl^{2}$ as $\gl\to+\infty$,
as said in the introduction.

\begin{proof}
For $\gl<0$, we use results from \cite{JansonLuczak2008}. The
parametrization there is slightly different, so we
define, given $\gl<0$, first $p:=1/n+\gl n^{-4/3}$ and then
$t:=-\log(1-p)=p+O(p^2)=n^{-1}+(\gl+o_n(1))n^{-4/3}$,
where $o_n(1)$ denotes a quantity that tends to 0 as $n\tooo$ for fixed
$\gl$.
Note that (for large $n$) $t<1/ n$ and 
$1-nt=(|\gl|+o_n(1))n^{-1/3}$. 
By \cite[Theorem 3.4]{JansonLuczak2008}, there exists a polynomial $p_k$ of
degree $2k-3$ such that 
\begin{equation}\label{pia}
  \begin{split}
	  \E_{n,p} \sum_i|C_i|^k 
& =n p_k\Bigpar{\frac{1}{1-nt}}\Bigpar{1+O\Bigpar{\frac{1}{n(1-nt)^3}}} \\
& =n p_k\bigpar{(|\gl|^{-1}+o_n(1))n^{1/3}}\bigpar{1+O\bigpar{(|\gl|+o_n(1))^{-3}}}.
  \end{split}
\end{equation}
Letting $a_k$ be the leading coefficient of $p_k$, so $p_k(x)\sim
a_kx^{2k-3}$ as $x\tooo$, we obtain by letting $n\tooo$ in \eqref{pia} 
and using
\eqref{eq:conv:exp} and \eqref{eq:conv:sum:sq},
\begin{align}
f_k(\gl)=\lim_{n\tooo} \Bigpar{n^{-2k/3} \E_{n,p} \sum_i|C_i|^k }
=a_k|\gl|^{-(2k-3)}\bigpar{1+O\bigpar{|\gl|^{-3}}}.
\end{align}
Note that this estimate holds uniformly in all $\gl<0$.
Finally, we note that $a_k=(2k-5)!!$, as remarked in 
\cite[after  (7.8)]{JansonLuczak2008}, and \eqref{foo-} follows.

For $\gl>0$, we use results from \cite{JansonSpencer2007}.
The idea is that as $\gl\to+\infty$, we approach the supercritical regime,
where there is a single giant component $C_1$ that dominates the sum
$\sum_i|C_i|^k$, and that $|C_1|\approx 2\gl$.

It is shown in \cite[Lemma 9.5]{JansonSpencer2007} that as $\gl\to+\infty$,
the intensity $\gLL(x)$ is well approximated by the density function of the
normal distribution $N(2\gl,2\gl^{-1})$ (except for small $x$), and 
\eqref{foo+} follows easily by \eqref{def:fk} and estimates as in the proof
of \cite[Lemma 9.5]{JansonSpencer2007}; we omit the details.
\end{proof}

For $\gl\to-\infty$, we can combine Theorem \ref{Tappa} 
with~\eqref{eq:fk:der} and obtain 
$\frac{\dx}{\dx\gl}f(\gl)=|\gl|^{-2}\bigpar{1+O(|\gl|^{-3})}$, and
similarly for larger $k$. This extends by induction to higher derivatives; 
the result shows that we can formally take any number of derivatives in
\eqref{foo-} (keeping the multiplicative error term $O(|\gl|^{-3})$).  

For $\gl\to+\infty$, the estimates 
for $\gLL(x)$ used in \cite{JansonSpencer2007} are not
precise enough to yield as precise results for derivatives 
(note that in~\eqref{eq:fk:der} we expect the leading terms 
of $f_{k+2}(\lambda)/2 \sim (2\lambda)^{k+2}/2$ and 
$\lambda f_{k+1} \sim \lambda (2\lambda)^{k+1}$ to cancel);
we conjecture that here too we can take derivatives formally 
in~\eqref{foo+}, but we have not tried to prove it. 
(This would require more precise estimates of $\gLL(x)$ and thus 
by \eqref{def:Lambda} and \eqref{def:wl} of
$\sum_{\ell\ge0} w_\ell x^{3\ell/2} =\E \exp\bigpar{x^{3/2}\Bex}$.
Such estimates can possibly be derived from the asymptotic expansions 
for the distribution of $\Bex$  in \cite{SJ203},
but we leave this as an open problem.)
We note only that~\eqref{main:der} and \eqref{rough:upper} imply 
$\frac{\dx}{\dx\gl}\log f(\gl)=O(\min\{|\gl|^{-1},1\})$ for all $\gl \in \RR$.

\section{Simple bounds for the susceptibility}\label{sec:susc}
In this appendix we give complete proofs of \eqref{eq:S-}--\eqref{eq:S+}, 
which we restate as the theorem below.
(The bounds are sharp up to constant factors when 
$np=1 \mp \eps$, $\eps=O(1)$ and $\eps^3 n \ge 1$.) 
\begin{theorem}\label{thm:susc'}
  \begin{thmenumerate}
  \item \label{ts'<}
For all $n \ge 1$, $p \in [0,1]$, and $\eps > 0$ satisfying $np \le 1-\eps$, 
\begin{equation}\label{eq:susc:subcr'}
\E_{n,p} |C(v)| \le \eps^{-1} .
\end{equation}
\item \label{ts'>}
There is a constant $D>0$ such that, for all $n \ge 1$, $p \in [0,1]$, 
and $\eps \ge 0$ satisfying $np \le 1+\eps$, 
\begin{equation}\label{eq:susc:supcr'}
\E_{n,p} |C(v)| \le D \max\set{\eps^2 n, n^{1/3}} .
\end{equation}
  \end{thmenumerate}
\end{theorem}
Part \ref{ts'<} is easy and well-known, and included for completeness. 
For part \ref{ts'>}, we do not know any reference with a short proof;
the bound is proved in \cite{BCHSS1} as a special case of a more general
and involved result. We give here a more direct argument which adapts
recent ideas from percolation theory~\cite{KN2009,vdHN2012} 
to the simpler~$G_{n,p}$ case.

We start by recalling some well-known branching processes 
results (we include proofs for completeness). 
Let $\bp_{n,p}$ denote a Galton--Watson branching process with 
$\Bin(n,p)$ offspring distribution, starting with a single 
individual, and let $|\bp_{n,p}|$ be its total size. 
We define $\bp_{\lambda}$ and $|\bp_{\lambda}|$ analogously, 
with~$\mathrm{Bin}(n,p)$ replaced by~$\mathrm{Po}(\lambda)$. 
\begin{lemma}\label{lem:bp}
  \begin{thmenumerate}
  \item \label{bp:dom}
For all $n \ge 1$, $p \in [0,1]$ and $\lambda \ge -n\log(1-p)$, 
$|\bp_{\lambda}|$ stochastically dominates $|\bp_{n,p}|$.
  \item \label{bp:tail}  
There exists a constant $C>0$ such that, for all 
$\lambda \ge 0$ and $k \ge 1$, we have  
\begin{equation}\label{bp}
  \Pr(k\le |\bp_{\lambda}| \le\infty) 
\le C\bigpar{\max\{\lambda-1,0\}+k^{-1/2}}.
\end{equation} 
  \end{thmenumerate}
\end{lemma}
\begin{proof}
\pfitemref{bp:dom} 
To prove that $|\bp_{\lambda}|$ stochastically dominates 
$|\bp_{n,p}|$, it suffices to show that $\mathrm{Po}(-n\log(1-p))$
stochastically dominates $\mathrm{Bin}(n,p)$. Taking $n$ 
independent couplings, it thus is enough to prove that 
$X \sim \mathrm{Po}(-\log(1-p))$ stochastically dominates
$Y \sim \mathrm{Bin}(1,p)$. This is immediate since 
$\Pr(X=0)=1-p = \Pr(Y =0)$, $\Pr(X \ge 1) = p = \Pr(Y=1)$ 
and $\Pr(Y \ge 2)=0$.

\pfitemref{bp:tail} 
Let $(\xi_i)_{i \ge 1}$ be a sequence of independent random variables 
with $\mathrm{Po}(\lambda)$ distribution. For all $k \ge 1$, using the 
classical Otter--Dwass formula~\cite{D1969} 
and Stirling's formula ($k! \ge \sqrt{2 \pi k} (k/e)^k$) we infer 
\begin{equation*}
\begin{split}
\Pr(|\bp_{\lambda}| =k) = \frac{\Pr(\xi_1 + \cdots + \xi_k = k-1)}{k} = \frac{\Pr(\mathrm{Po}(k\lambda) = k-1)}{k} = \frac{e^{-\lambda k}}{k \lambda} \frac{(\lambda k)^k}{k!} \le \frac{(\lambda e^{1-\lambda})^k}{\sqrt{2 \pi} k^{3/2}\lambda} \le \frac{e}{\sqrt{2\pi} k^{3/2}} , 
\end{split}
\end{equation*}
where  the last inequality follows by noting 
$\lambda e^{1-\lambda} \le 1$ 
and thus $\xpar{\lambda e^{1-\lambda}}^k \le \gl e^{1-\gl}\le e\gl$. 
Summing this inequality, we see that there 
is a constant $C$ such that 
\begin{equation}
  \label{bp1}
\Pr(k\le |\bp_{\lambda}| <\infty) \le C k^{-1/2}.
\end{equation}

It is easy to see that $\rho := \Pr(|\bp_{\lambda}| =\infty)$ satisfies 
$1-\rho = \E (1-\rho)^{\mathrm{Po}(\lambda)} = e^{-\lambda \rho}$. Using 
Taylor series we infer $\lambda\rho = -\log(1-\rho) \ge \rho + \rho^2/2$, 
so that either $\rho=0$ or $0 < \rho \le 2(\lambda -1)$, 
which together with \eqref{bp1} completes the proof of \eqref{bp}. 
\end{proof}

Given a graph $H$, 
we write $C_H(v)$ for the component containing the vertex~$v$ in $H$, and  
$\dist_H(v,w)$ for the length of the shortest path 
between~$v$ and~$w$ in $H$ (setting $\dist_H(v,w)=\infty$ 
if there is no such path). 
Define $B_{H}(v,r) = \{w \in V(H) : \dist_H(v,w) \le r\}$ and 
$\partial B_H(v,r) = \{w \in V(H) : \dist_H(v,w) = r\}$. 
\begin{lemma}\label{lem:Gamma}
There is a constant $C >0$ such that, for all $n \ge 1$, $p \in [0,1]$, and 
$\eps >0$ satisfying $np \le 1+\eps$ and $\eps n \ge 1$, the following 
holds for all\/ $1 \le r \le \ceil{\eps^{-1}}$: 
\begin{equation}\label{Gamma}
\Gamma(r) := \max_{G \subseteq K_n} \max_{v \in V(G)} \Pr(\partial B_{G_{p}}(v,r) \neq \emptyset) \le C r^{-1} .
\end{equation}
\end{lemma}
\begin{proof}
Assuming $C \ge 9$, note that for $\eps \ge 1/9$ inequality~\eqref{Gamma} 
holds trivially for all\/ $1 \le r \le \ceil{\eps^{-1}}$. 
It thus suffices to consider the case $\eps \le 1/9$. 
Let $\gl:=-n\log(1-p)$.
Note that $p \le (1+\eps)/n \le 2\eps < 1/2$ 
and thus $\gl\le np(1+p) \le 1+4\eps$. 
Let $k_0 \ge 1$ satisfy $3^{k_0-1} \le \ceil{\eps^{-1}} < 3^{k_0}$. 

It is well-known and easy to see that 
for any subgraph $G \subseteq K_n$,
$|C_{G_p}(v)|\le |C_{G_{n,p}}(v)|$
is stochastically dominated by
$|\bp_{n,p}|$.
By Lemma~\ref{lem:bp} it follows that there exists a constant $B>0$ 
such that, for all $1 \le K \le 9^{k_0}$, 
\begin{equation}\label{Gamma:dom}
\max_{G \subseteq K_n} \max_{v \in V(G)} \Pr(|C_{G_p}(v)| \ge K) 
\le  \Pr(|\bp_{n,p}| \ge K) 
\le  \Pr(|\bp_{\gl}| \ge K) 
\le \Pr(|\bp_{1+4\eps}| \ge K) \le B K^{-1/2} .
\end{equation}

Mimicking~\cite[Section~3.2]{KN2009}, we now show by induction 
on~$k \in \NN$ that $D:=(3^{3} + B)^3$ satisfies 
\begin{equation}\label{gamma:ind}
\Gamma(3^k) \le D 3^{-k} \qquad \text{for all $0 \le k \le k_0$,} 
\end{equation} 
which readily implies~\eqref{Gamma} with $C=3D$ 
(for any $1 \le r \le \ceil{\eps^{-1}}$ there is $1 \le k \le k_0$ with 
$3^{k-1} \le r < 3^k$, so $\Gamma(r) \le \Gamma(3^{k-1}) \le 3D r^{-1}$ follows).
The base case $k=0$ holds trivially since since $\Gamma(1) \le 1 \le D$. 

For the induction step, let $1 \le k \le k_0$
and assume that \eqref{gamma:ind} holds for $k-1$. 
Fix $G \subseteq K_n$ and $v \in V(G)$.
Set $\delta := D^{-4/3} \le 1$.
Then, by \eqref{Gamma:dom} we see that  
\begin{equation}\label{gamma:ind:case}
\Pr\bigpar{\partial B_{G_{p}}(v,3^{k}} \neq \emptyset) 
\le \Pr\bigpar{\partial B_{G_{p}}(v,3^{k}) \neq \emptyset \text { and } 
|C_{G_p}(v)| < \delta 9^k} + B \delta^{-1/2} 3^{-k} .
\end{equation}
By the pigeonhole principle, 
if $\partial B_{G_{p}}(v,3^{k}) \neq \emptyset$ 
and $|C_{G_p}(v)| < \delta 9^k$, then at least one level $j$ with
$3^{k-1} \le j \le 2\cdot 3^{k-1}$ satisfies 
$0 < |\partial B_{G_{p}}(v,j)| \le \delta 3^{k+1}$; 
let~$J$ denote the smallest such level. 
(If no such $j$ exists, let $J:=\infty$.) 
Note that, 
for any given non-empty
sets of vertices $W\subseteq[n]$ and $\partial W\subseteq W$,
using a breadth-first-search neighbourhood 
exploration algorithm, we can determine 
whether $B_{G_{p}}(v,J)=W$ and $\partial B_{G_{p}}(v,J)=\partial W$ 
by testing the status (in~$G_p$) only of edges 
with at least one endpoint in $W \setminus \partial W$.
Furthermore, if this event holds, then this determines $J$ and $J<\infty$. 
Consequently,
if $H$ is the induced subgraph of $G$ with vertex set 
$[n]\setminus(W \setminus \partial W)$, then
after conditioning on $B_{G_{p}}(v,J)=W$ and
$\partial B_{G_{p}}(v,J)=\partial W$,
the remaining random graph 
$G_p \cap H$ has the same distribution as the \emph{unconditional} random graph~$H_p$. 
Furthermore, by construction, the shortest path in $G_p$ 
from~$\partial W$ to $\partial B_{G_{p}}(v,3^{k})$ 
contains only edges in $H$, so $\partial B_{G_{p}}(v,3^k) \neq \emptyset$ 
implies the existence of a vertex~$w \in \partial W$ with $\partial B_{G_p \cap H}(w,3^{k}-J) \neq \emptyset$. 
Combining $|\partial B_{G_{p}}(v,J)| \le \delta 3^{k+1}$ and $3^{k}-J \ge 3^{k-1}$ 
with $H \subseteq G \subseteq K_n$ and $\partial W \subseteq V(H)$, it
follows that 
\begin{equation}\label{gamma:ind:cond1}
\begin{split}
\Pr\bigpar{\partial B_{G_{p}}(v,3^{k}) \neq \emptyset \mid B_{G_{p}}(v,J)=W, \: \partial B_{G_{p}}(v,J)=\partial W} & 
\le \sum_{w \in \partial W} 
\Pr\bigpar{\partial B_{H_p}(w,3^{k-1}) \neq \emptyset} \\
& \le \delta 3^{k+1} \cdot \Gamma(3^{k-1}) .
\end{split}
\end{equation}
Since the bound in \eqref{gamma:ind:cond1} does not depend on $W$ and
$\partial W$, it follows that
\begin{equation*}\label{cond:J}
\Pr\bigpar{\partial B_{G_{p}}(v,3^{k}) \neq \emptyset\mid J<\infty}  
\le \delta 3^{k+1} \Gamma(3^{k-1}) .
\end{equation*}
Consequently, using the induction hypothesis (and recalling that $J \ge 3^{k-1}$), 
\begin{equation}\label{gamma:ind:cond2}
\begin{split}
\Pr\bigpar{\partial B_{G_{p}}(v,3^{k}) \neq \emptyset \text { and }
  |C_{G_p}(v)| < \delta 9^k} 
& \le 
\Pr\bigpar{\partial B_{G_{p}}(v,3^{k}) \neq \emptyset \text { and }
J<\infty}
\\& 
\le \delta 3^{k+1} \Gamma(3^{k-1}) \cdot \Pr(J<\infty)
\\& 
\le
\delta 3^{k+1} \Gamma(3^{k-1}) \cdot 
\Pr\bigpar{\partial B_{G_{p}}(v,3^{k-1}) \neq \emptyset} 
\\& 
\le
\delta 3^{k+1} \Gamma(3^{k-1})^2 
 \le \delta 3^{k+1} (D 3^{1-k})^2 .
\end{split}
\end{equation}
After inserting~\eqref{gamma:ind:cond2} into~\eqref{gamma:ind:case}, 
by recalling $\delta = D^{-4/3}$ and $D=(3^{3} + B)^3$ we infer 
\[
\Pr\bigpar{\partial B_{G_{p}}(v,3^{k}) \neq \emptyset} 
\le \bigpar{3^{3}\delta D^2 + B \delta^{-1/2}} 3^{-k} 
= (3^{3} + B) D^{2/3} 3^{-k} = D 3^{-k} ,
\]
completing the proof of the induction step 
(since $G \subseteq K_n$ and $v \in V(G)$ were arbitrary). 
\end{proof}
\begin{proof}[Proof of Theorem~\ref{thm:susc'}]
\pfitemref{ts'<}
Since~$|\bp_{n,p}|$ stochastically dominates~$|C(v)|$, 
using $np \le 1-\eps$ we infer 
\[
\E |C(v)| \le \E |\bp_{n,p}| = \sum_{j \ge 0} (np)^j = (1-np)^{-1} \le \eps^{-1} .
\]
\pfitemref{ts'>}
Suppose first that $\eps \le1$ and $\eps^3 n \ge 1$ (the upper bound 
$\eps \le 1$ conveniently ensures $\eps \ceil{\eps^{-1}} \le2$).
We set $r:=\ceil{\eps^{-1}}$, and proceed by a case distinction similar 
to~\cite[Lemma~2.3]{vdHN2012}. Observe that 
\begin{equation}
  \label{cel}
\Enp |C(v)| = \sum_{w \in [n]} \Pr_{n,p}(w \in C(v)) 
= \sum_{w \in [n]} \bigl[\Pr(\dist_{G_{n,p}}(v,w) \le 2r) + \Pr(2r < \dist_{G_{n,p}}(v,w) < \infty)\bigr] .
\end{equation}
Since $|B_{\bp_{n,p}}(v,r)|$ stochastically dominates 
$|B_{G_{n,p}}(v,r)|$, using $np \le 1+\eps$ and $\eps \le 1$ we deduce 
\begin{equation}
\begin{split}
  \sum_{w \in [n]} \Pr\bigpar{\dist_{G_{n,p}}(v,w) \le 2r} 
& = \E |B_{G_{n,p}}(v,2r)| \le \E |B_{\bp_{n,p}}(v,2r)| \\
& \le \sum_{0 \le j \le  2r} (1+\eps)^j 
\le \eps^{-1} (1+\eps)^{2r+1} 
\le \eps^{-1} e^{5}
.
\end{split}
\end{equation}
Note that $2r < \dist_{G_{n,p}}(v,w) < \infty$ implies 
$B_{G_{n,p}}(v,r) \cap B_{G_{n,p}}(w,r) = \emptyset$ and 
$\partial B_{G_{n,p}}(v,r),\,\partial B_{G_{n,p}}(w,r) \neq \emptyset$. 
By conditioning on $B_{G_{n,p}}(v,r)$, and letting $H:=G_{n,p}\setminus
B_{G_{n,p}}(v,r)$, it follows that, similarly to
\eqref{gamma:ind:cond1},
\begin{equation*}\label{cel1}
\Pr\bigpar{2r < \dist_{G_{n,p}}(v,w) < \infty\mid
 B_{G_{n,p}}(v,r) }
\le
\Pr(\partial B_{H}(w,r) \neq \emptyset) 
\indic{\partial B_{G_{n,p}}(v,r)\neq\emptyset}
\le \Gamma(r) \indic{\partial B_{G_{n,p}}(v,r)\neq\emptyset}
 \end{equation*}
and consequently by taking the expectation and
using Lemma~\ref{lem:Gamma},
\begin{equation}\label{cel2}
\Pr\bigpar{2r < \dist_{G_{n,p}}(v,w) < \infty}
\le \Gamma(r) \cdot \Pr\bigpar{\partial B_{G_{n,p}}(v,r)\neq\emptyset}
\le \Gamma(r)^2
\le (C r^{-1})^2 
\le C^2\eps^2.
 \end{equation}
By \eqref{cel}--\eqref{cel2},
and $\eps^3 n \ge 1$, there thus is a constant $D=D(C)$ such that  
\begin{equation}\label{eq:susc:supcr:eps3n}
\E_{n,p} |C(v)|  \le e^{5}\eps^{-1} + C^2 \eps^2 n \le D \eps^2n.
\end{equation}

This proves \eqref{eq:susc:supcr'} when $\eps \le 1$ and $\eps^3n\ge1$.
When $\eps \ge 1$, 
the bound in
\eqref{eq:susc:supcr'} holds 
trivially 
 (since $|C(v)| \le n$),
assuming as we may $D \ge 1$.

In the remaining case  $\eps^3 n < 1$,
we observe that $np \le 1+\eps \le 1+n^{-1/3}$. 
Hence~\eqref{eq:susc:supcr:eps3n} with $\eps:=n^{-1/3}$ implies 
$\E_{n,p} |C(v)| \le D n^{1/3}$, 
completing the proof of \eqref{eq:susc:supcr'}. 
\end{proof}

\section{Higher derivatives of the susceptibility} 
\label{AppAhigher}
In this appendix we extend the method of proof from
Section~\ref{sec:convergence} 
to higher derivatives,
using arguments from 
\cite{JansonLuczak2008}.
The key fact is that, extending~\eqref{SnCi:conv:der},  
any mixed moment of~$X_{n,k}$ defined in~\eqref{def:Xnkq}, $k\ge2$, has a
derivative that can be expressed as a linear combination of such moments,
and thus by induction the same holds for higher derivatives as well.
We illustrate the general method by some  examples, leaving the details
in the general case to the reader. 
For notational convenience, we write 
\begin{equation*}
\DD:=n^{-4/3}(1-p)\frac{\dx}{\dx p} .  
\end{equation*}
Note that $\DD=n^{-4/3}\frac{\dx}{\dx t}$ for the parametrization 
$p=1-e^{-t}$ used in \cite{JansonLuczak2008}.
Note also that the factor $1-p$, which is needed in the exact formulas below,
disappear asymptotically, since $p=o(1)$, and that 
apart from this factor, $\DD=\frac{\dx}{\dx\lambda}$ for our usual
parametrization $p=n^{-1}+\lambda n^{-4/3}$.

First, consider $\DD\bigpar{\Enp (X_{n,k})}$ for an arbitrary $k\ge2$.
As noted in~\cite[(3.1)]{JansonLuczak2008}, if~$v\notarrow w$, then 
adding the edge~$vw$ to the graph increases $\sum_i|C_i|^k$ by 
\begin{equation}\label{perdita}
\Delta_{vw}\Bigpar{ \sum_i|C_i|^k}
:=(|C(v)|+|C(w)|)^k-|C(v)|^k-|C(w)|^k
=\sum_{\ell=1}^{k-1}\binom{k}{\ell}|C(v)|^\ell|C(w)|^{k-\ell} .
\end{equation}
(And, trivially, the change
$\Delta_{vw}(\sum_i|C_i|^k) = 0$ if~$v\leftrightarrow w$.) 
Recalling $X_{n,k} = n^{-2k/3}\sum_i|C_i|^k$, see~\eqref{def:Xnkq}, 
it follows by a modification of the argument leading 
to~\eqref{eq:SQ:sum} that (similar to~\cite[Theorem~2.32]{Grimmett1999}), 
with a factor~$\frac12$ because each edge is counted twice, 
\begin{equation}\label{eq:Xnk:der}
  \begin{split}
	\DD\bigpar{\Enp(X_{n,k})}
&=\tfrac12 n^{-2(k+2)/3} \sum_{v,w \in [n]} \Enp\lrpar{
 \indic{vw\not\in G_{n,p}} \Delta_{vw}\Bigpar{ \sum_i|C_i|^k}}
\\
&=\tfrac12 n^{-2(k+2)/3} \Enp\lrpar{
 \sum_{\ell=1}^{k-1}\binom{k}{\ell}\sum_{i\neq j}|C_i|^{\ell+1}|C_j|^{k-\ell+1}}
\\
&=\tfrac12  
 \sum_{\ell=1}^{k-1}\binom{k}{\ell}\Enp\bigpar{X_{n,\ell+1}X_{n,k-\ell+1}}
-(2^{k-1}-1)\Enp(X_{n,k+2}).
  \end{split}
\end{equation}
Theorem \ref{thm:scaling}\ref{thm:scaling-ii} extends to mixed moments,
because the convergence \eqref{eq:conv:sum} holds jointly for different
$k\ge2$ (by the same proof), 
and the uniform moment bound \eqref{bound:sum:sq} extends to mixed moments
by \Holder's inequality.
 Thus we obtain from \eqref{eq:Xnk:der},
if $p =1/n + \bigl(\lambda+o(1)\bigr) n^{-4/3}$,
\begin{equation}\label{hermione}
  \begin{split}
	\DD\bigpar{\Enp(X_{n,k})}
\to
\tfrac12  
 \sum_{\ell=1}^{k-1}\binom{k}{\ell}\E\bigpar{W_{\gl,\ell+1}W_{\gl,k-\ell+1}}
-(2^{k-1}-1)\E(W_{\gl,k+2}).
  \end{split}
\end{equation}
The special case $k=2$ is given above in \eqref{SnCi:conv:der}.

For higher moments, we give for notational convenience
just one example of the method. 
Adding an edge~$vw$ with $v\notarrow w$ increases
$X_{n,2}=n^{-4/3}\sum_i|C_i|^2$
by $\Delta(X_{n,2})=2n^{-4/3}|C(v)||C(w)|$, see \eqref{perdita}, 
and thus increases~$X_{n,2}^2$ by
\begin{equation}
  \begin{split}
	\Delta\bigpar{X_{n,2}^2}
&=\bigpar{X_{n,2}+\Delta X_{n,2}}^2-X_{n,2}^2
=2X_{n,2}\Delta X_{n,2}+(\Delta X_{n,2})^2
\\&
=4n^{-4/3}X_{n,2}|C(v)||C(w)|
+4n^{-8/3}|C(v)|^2|C(w)|^2,
  \end{split}
\end{equation}
leading to
\begin{equation}
  \begin{split}
\DD\bigpar{\Enp\bigpar{X_{n,2}^2}}
&=\tfrac12\Enp \Bigpar{4n^{-8/3}X_{n,2}\sum_{i\neq j}|C_i|^2|C_j|^2
+4n^{-12/3}\sum_{i\neq j}|C_i|^3|C_j|^3}
\\&
=2\Enp \bigpar{X_{n,2}\bigpar{X_{n,2}^2-X_{n,4}}}
+2\Enp \bigpar{X_{n,3}^2-X_{n,6}}
\\&
=2\Enp\bigpar{X_{n,2}^3}-2\Enp\bigpar{X_{n,4}X_{n,2}}+2\Enp \bigpar{X_{n,3}^2}
-2\Enp\bigpar{X_{n,6}}.
  \end{split}
\end{equation}
This together with the cases $k=2$ and $k=4$ of \eqref{eq:Xnk:der} yield,
after simplifications,
\begin{equation}
  \begin{split}
(\DD)^2\Bigpar{\Enp\bigpar{X_{n,2}}}
&= \DD\Bigpar{\Enp\bigpar{X_{n,2}^2}-\Enp\bigpar{X_{n,4}}}
\\&
=2\Enp\bigpar{X_{n,2}^3}-6\Enp\bigpar{X_{n,4}X_{n,2}}-\Enp \bigpar{X_{n,3}^2}
+5\Enp\bigpar{X_{n,6}}
  \end{split}
\end{equation}
and thus we obtain, if 
$p =1/n + \bigl(\lambda+o(1)\bigr) n^{-4/3}$,
\begin{equation}
  \begin{split}
(\DD)^2\Bigpar{\Enp\bigpar{X_{n,2}}}
\to
2\E\bigpar{W_{\gl,2}^3}-6\E\bigpar{W_{\gl,4}W_{\gl,2}}-\E \bigpar{W_{\gl,3}^2}
+5\E\bigpar{W_{\gl,6}}.
  \end{split}
\end{equation}
The general case is similar. 
In particular, this leads to the following extension of Theorem \ref{thm:main2}.

\begin{theorem}\label{Tconv:higher}
Define the infinitely differentiable function $f=f_2: \RR \to (0,\infty)$ as
in~\eqref{def:fk}.
Given $\lambda \in \RR$, for $p =1/n + \bigl(\lambda+o(1)\bigr) n^{-4/3}$ we
have, as $n \to \infty$, for every fixed $m$,
\begin{align}
\label{main:der:f}
&& &&{n^{-(4m+1)/3}} \frac{\dx^m}{\dx p^m} \goodchi_{K_n}(p)
& \to \frac{\dx^m}{\dx \lambda^m} f(\lambda),
&& (m\ge0),&&&&
\\
\label{main:der:logf}
&&
&&{n^{-4m/3}} \frac{\dx^m}{\dx p^m} \log \goodchi_{K_n}(p)
& \to \frac{\dx^m}{\dx \lambda^m} \log f(\lambda)
&& (m\ge1).
\end{align}
Moreover, if $p =1/n + \lambda n^{-4/3}$, then the convergence is uniform for $\gl$ in any compact set~$[\lambda_1,\lambda_2] \subset \RR$.
\end{theorem}
\begin{proof}
For $p =1/n + \bigl(\lambda+o(1)\bigr) n^{-4/3}$,
the argument above shows that for every $m\ge0$,
\begin{equation}\label{eq:conv:Xn2:0}
(\DD)^m\Bigpar{\Enp\bigpar{X_{n,2}}}\to g_m(\gl)
\end{equation}
for some function~$g_m(\gl)$, with  
\begin{equation}\label{eq:gof}
g_0(\lambda)=f(\lambda).
\end{equation} 
Recalling the definition of~$\DD$, using~\eqref{eq:conv:Xn2:0}, the 
product rule, and induction, it is easy to see that for every $m \ge 0$
there are constants~$c_{j,m} \in \RR$ 
with $c_{m,m}=1$
such that 
\begin{equation}\label{eq:conv:Xn2:1}
(\DD)^m\Bigpar{\Enp\bigpar{X_{n,2}}} 
= \sum_{0 \le j \le m} c_{j,m} (1-p)^j n^{-4m/3} 
  \frac{\dx^j}{\dx p^j}\Bigpar{\Enp\bigpar{X_{n,2}}} .
\end{equation}
Combining~\eqref{eq:conv:Xn2:0} and~\eqref{eq:conv:Xn2:1} with $p=o(1)$, 
by another induction on $m \ge 0$ we now infer 
\begin{equation}\label{eq:conv:Xn2}
n^{-4m/3}\frac{\dx^m}{\dx p^m}\Bigpar{\Enp\bigpar{X_{n,2}}}\to g_m(\gl),
\end{equation}
since in~\eqref{eq:conv:Xn2:1} any summand with $j < m$ is~$O(n^{-4m/3} \cdot n^{4j/3})=o(1)$ by the induction hypothesis. 

We now change parametrization and define 
$h_n(\gl):=\Enp\bigpar{X_{n,2}}$ for $p =1/n + \lambda n^{-4/3}$, 
so that~\eqref{eq:conv:Xn2} translates into 
\begin{equation}\label{hg}
\frac{\dx^m}{\dx \gl^m}h_n(\gl) 
= n^{-4m/3}\frac{\dx^m}{\dx p^m} h_n(\gl) \to g_m(\gl).
\end{equation}
Lemma \ref{Luc} shows that $g_m(\gl)$ is continuous,
and that \eqref{hg} holds uniformly for $\gl$ in any
compact set.
Hence we can integrate, as in
\eqref{eq:scaling:diff:f}--\eqref{eq:scaling:diff:intg}, and obtain
$g_m(\gl_2)-g_m(\gl_1)=\int_{\gl_1}^{\gl_2} g_{m+1}(\gl)\dd\gl$
whenever $\gl_1<\gl_2$, and thus $g_{m+1}(\gl)=\frac{\dx}{\dx\gl}g_m(\gl)$.
By induction and~\eqref{eq:gof}, we infer, for every $m\ge0$,
\begin{equation}\label{eq:func:fm}
g_m(\gl)=\frac{\dx^m}{\dx \gl^m}g_0(\gl)=\frac{\dx^m}{\dx \gl^m}f(\gl).
\end{equation}
By \eqref{hg} and \eqref{eq:func:fm}, 
\begin{equation}\label{hgf1}
  \frac{\dx^m}{\dx \gl^m}h_n(\gl)\to 
g_m(\gl)=
 \frac{\dx^m}{\dx \gl^m}f(\gl),
\qquad m\ge0,
\end{equation}
and thus also, by expanding the derivatives of the logarithms on both sides
and applying \eqref{hgf1} to each term,
\begin{equation}\label{hgf2}
  \frac{\dx^m}{\dx \gl^m}\log h_n(\gl)\to 
 \frac{\dx^m}{\dx \gl^m}\log f(\gl),
\qquad m\ge1.
\end{equation}
Moreover, the convergence in \eqref{hgf1} and \eqref{hgf2} is uniform on any
compact set. 
Consequently, \eqref{hgf1}--\eqref{hgf2} hold also with $h_n(\gl_n)$ on the
left-hand side, for any sequence $\gl_n\to \gl$, see Remark \ref{Rluc}.
The result \eqref{main:der:f}--\eqref{main:der:logf} now follows 
for any $p =1/n + \bigl(\lambda+o(1)\bigr) n^{-4/3}$
by taking $\gl_n:=n^{4/3}(p-1/n)=\gl+o(1)$.
\end{proof}

\small
\begin{spacing}{0.9}

\end{spacing}
\normalsize

\end{document}